\documentclass[12pt]{amsart}
\usepackage[margin=1.35in]{geometry}
\usepackage{amscd,amsmath,amsxtra,amsthm,amssymb,stmaryrd,xr,mathrsfs,mathtools,enumerate,commath, comment, mathtools}
\usepackage{stmaryrd}
\usepackage{multirow}
\usepackage{xcolor}
\usepackage{commath}
\usepackage{comment}
\usepackage{tikz-cd}
\usepackage{longtable} 
\usepackage{pdflscape} 
\usepackage{booktabs}
\usepackage{hyperref}
\definecolor{vegasgold}{rgb}{0.77, 0.7, 0.35}
\definecolor{darkgoldenrod}{rgb}{0.72, 0.53, 0.04}
\definecolor{gold(metallic)}{rgb}{0.83, 0.69, 0.22}
\hypersetup{
 colorlinks=true,
 linkcolor=darkgoldenrod,
 filecolor=brown,      
 urlcolor=gold(metallic),
 citecolor=darkgoldenrod,
 }
\newtheorem{lthm}{Theorem}

\usepackage[all,cmtip]{xy}

\DeclareFontFamily{U}{wncy}{}
\DeclareFontShape{U}{wncy}{m}{n}{<->wncyr10}{}
\DeclareSymbolFont{mcy}{U}{wncy}{m}{n}
\DeclareMathSymbol{\Sh}{\mathord}{mcy}{"58}
\usepackage[T2A,T1]{fontenc}
\usepackage[OT2,T1]{fontenc}

\newtheorem{theorem}{Theorem}[section]
\newtheorem{lemma}[theorem]{Lemma}
\newtheorem{ass}[theorem]{Assumption}
\newtheorem*{theorem*}{Theorem}
\newtheorem*{ass*}{Assumption}
\newtheorem{definition}[theorem]{Definition}
\newtheorem{corollary}[theorem]{Corollary}
\newtheorem{remark}[theorem]{Remark}
\newtheorem{example}[theorem]{Example}
\newtheorem{conjecture}[theorem]{Conjecture}
\newtheorem{proposition}[theorem]{Proposition}
\newtheorem{question}[theorem]{Question}

\newcommand{\ord}{\mathrm{ord}}

\newcommand{\Z}{\mathbb{Z}}
\newcommand{\Q}{\mathbb{Q}}
\newcommand{\F}{\mathbb{F}}

\newcommand{\cO}{\mathcal{O}}

\newcommand{\Sel}{\mathrm{Sel}}

\newcommand{\cyc}{\mathrm{cyc}}

\newcommand{\op}[1]{\operatorname{#1}}

 \DeclareMathSymbol{\sha}{\mathord}{mcy}{"58}
 \makeatletter
\newcommand{\mylabel}[2]{#2\def\@currentlabel{#2}\label{#1}}
\makeatother

\numberwithin{equation}{section}

\begin{document}

\title[Hilbert's tenth problem for $\mathbb{Z}_p$-extensions]{Hilbert's tenth problem for families of $ \mathbb{Z}_p $-extensions of imaginary quadratic fields}
\author[K.~M\"uller]{Katharina M\"uller}
\address[Müller]{Institut für Theoretische Informatik, Mathematik und Operations Research, Universität der Bundeswehr München, Werner-Heisenberg-Weg 39, 85577 Neubiberg, Germany}
\email{katharina.mueller@unibw.de}

\author[A.~Ray]{Anwesh Ray}
\address[Ray]{Chennai Mathematical Institute, H1, SIPCOT IT Park, Kelambakkam, Siruseri, Tamil Nadu 603103, India}
\email{anwesh@cmi.ac.in}


\maketitle

\begin{abstract}
Via a novel application of Iwasawa theory, we study Hilbert's tenth problem for number fields occurring in $\Z_p$-towers of imaginary quadratic fields $K$. For an odd prime $p$, the lines $(a,b) \in \mathbb{P}^1(\mathbb{Z}_p)$ are identified with $\mathbb{Z}_p$-extensions $K_{a,b}/K$. Under certain conditions on $K$ that involve explicit elliptic curves, we identify a line $(a_0,b_0) \in \mathbb{P}^1(\Z/p\Z)$ such that for all $(a,b) \in \mathbb{P}^1(\mathbb{Z}_p)$ with $(a,b)\not\equiv(a_0,b_0)\pmod p$, Hilbert's tenth problem has a negative answer in all finite layers of $K_{a,b}$. Using results of Bhargava et al., we prove unconditionally that a positive proportion of imaginary quadratic fields meet our criterion when $p=3$. For $p=11,13,31,37$, the analogous conclusions obtained from the rank-zero twist families of Kriz--Li are conditional on the vanishing of the $p$-primary Tate--Shafarevich groups for a positive relative proportion of those twists.
\end{abstract}

\section{Introduction}
\subsection{Context and motivation} Hilbert's tenth problem inquires whether there exists a universal algorithm to determine the solvability of any given Diophantine equation over the integers $\mathbb{Z}$. This problem was resolved negatively by Matijasevich \cite{matiyasevich}, who demonstrated that no such algorithm exists. Extending this inquiry to number fields, we can ask whether a procedure can be devised to decide the solvability of Diophantine equations over the ring of integers $\cO_K$ of a number field $K$.

It is established that if the ring of rational integers $\mathbb{Z}$ can be described as a Diophantine set within $\cO_K$, then the analogue of Hilbert's tenth problem for $\cO_K$ also has a negative solution. The Denef-Lipshitz conjecture posits that $\mathbb{Z}$ is indeed a Diophantine subset of $\cO_K$ for any number field $K$. This conjecture has been validated in several specific instances:
\begin{itemize}
    \item When $K$ is totally real \cite{denef1980diophantine}.
    \item When $K$ is a quadratic extension of a totally real number field \cite{denef1978diophantine}.
    \item When $K$ has exactly one complex place \cite{shlapentokh1989extension}, \cite{pheidas1988hilbert}.
    \item When $K/\mathbb{Q}$ is abelian \cite{shapiro1989diophantine}.
\end{itemize}
It is natural to seek to prove the conjecture for families of number fields for which the above results do not apply. Mazur and Rubin \cite{mazur2018diophantine} introduced the notion of \emph{diophantine stability} for a variety defined over a number field. By studying this notion for abelian varieties, they were able to prove that Hilbert's tenth problem has a negative answer for certain large families of number fields. Garcia-Fritz and Pasten \cite{Gracia-Fitzen-Pasten} confirmed the Denef--Lipshitz conjecture for fields of the form $\mathbb{Q}(p^{1/3}, \sqrt{-q})$, with $p$ and $q$ being primes from specific sets characterized by positive Tchebotarev density. Recently, Kundu, Lei, and Sprung \cite{kundu-lei-sprung} extended these results, proving the conjecture for a broader family of number fields $\mathbb{Q}(p^{1/3}, \sqrt{Dq})$, where $D$ ranges over a finite set of integers, and $p$ and $q$ are primes within explicit sets with positive Tchebotarev density.

The number fields examined in these studies are all of degree less than 6. This raises the question of whether the conjecture can be proven for families of non-abelian number fields with unbounded degrees. Partial progress has been made by the second named author \cite{ray-hilbert10}. Consider a number field $K$ and its cyclotomic $\mathbb{Z}_p$-extension $K_{\op{cyc}}$, with $K_n$ as the intermediate extension of degree $p^n$, where $p$ is an odd prime. If an elliptic curve $E/K$ satisfies certain local conditions at $p$, then $\mathbb{Z}$ is a Diophantine subset of $\cO_{K_n}$ for all $n$. Starting from an imaginary quadratic field $K$ and examining the anticyclotomic $\mathbb{Z}_p$-extension, similar results were obtained by the second named author and Weston for the anticyclotomic $\mathbb{Z}_p$-extension of $K$ \cite{ray-weston}. This utilized the theory of congruences between Galois representations, and the techniques of Greenberg and Vatsal \cite{greenbergvatsal}, who study the effect of such congruences in Iwasawa theory. This leads to explicit examples, for instance, it is shown that in all finite layers of the anticyclotomic $\Z_3$-extension of $\Q(\sqrt{-5})$, Hilbert's tenth problem has a negative answer.
\subsection{Main results}
\par The goal of the present paper is to start with an elliptic curve over an imaginary quadratic field $K$. Since $K$ is abelian, Hilbert's tenth problem has a negative solution in all number fields contained in the cyclotomic $\Z_p$-extension of $K$. Note that in order to obtain interesting new results, we wish to consider a family of $\Z_p$-extensions of $K$, other than the cyclotomic one. 
Before presenting our main results, we introduce some additional notation. Let $ E_{/\mathbb{Q}} $ be an elliptic curve and $ K_\infty $ the $\mathbb{Z}_p^2$-extension of $ K $. Define $ \Gamma_\infty := \mathrm{Gal}(K_\infty/K) $ and $ \Gamma_{\mathrm{cyc}} := \mathrm{Gal}(K_{\mathrm{cyc}}/K) $. Let $ \sigma, \tau \in \Gamma_\infty $ be topological generators with:
\[
\begin{split}
    &\overline{\langle \sigma \rangle} = \ker\left(\Gamma_\infty \rightarrow \Gamma_{\mathrm{cyc}}\right),\\
    & \overline{\langle \tau \rangle} = \ker\left(\Gamma_\infty \rightarrow \Gamma_{\mathrm{ac}}\right).
\end{split}
\]
Set $ X+1 = \sigma $ and $ Y+1 = \tau $, allowing us to identify $ \Lambda_\infty $ with the power series ring $ \mathbb{Z}_p\llbracket X,Y\rrbracket $. Consider $ \mathbb{P}^1(\mathbb{Z}_p) $, the space of lines in $ \mathbb{Z}_p^2 $. Each line has a representative $ (a,b) \in \mathbb{Z}_p^2 $ with $ p \nmid a $ or $ p \nmid b $, uniquely determined up to a unit in $ \mathbb{Z}_p $. For $ (a,b) \in \mathbb{P}^1(\mathbb{Z}_p) $, there is a unique $\mathbb{Z}_p$-extension $ K_{a,b}/K $ such that:
\[
\overline{\langle \sigma^a \tau^b \rangle} = \ker\left(\Gamma_\infty \rightarrow \mathrm{Gal}(K_{a,b}/K)\right).
\]
For $n\in \Z_{\geq 0}$, let $K_{a,b}^{(n)}/K$ be the unique extension contained in $K_{a,b}$ such that $[K_{a,b}^{(n)}:K]=p^n$. Denote by $E^{(K)}$ the twist of $E$ over $K$. The $p$-primary Selmer group of $E$ over $\Q$ is denoted $\op{Sel}_{p^\infty}(E/\Q)$. Given a prime $\ell$, set $c_\ell(E/\Q)$ to be the Tamagawa number of $E_{/\Q_\ell}$. Assuming that $E$ has good reduction at $p$, take $\widetilde{E}(\F_p)$ to be the group of $\F_p$-points of the reduction of $E$ at $p$. 
\begin{lthm}[Theorems \ref{main thm 2} and \ref{main thm 3}]
    Let $E_{/\Q}$ be an elliptic curve, $p$ an odd prime and $K$ be an imaginary quadratic field. Assume that the following conditions hold
    \begin{enumerate}
        \item $\op{corank}_{\Z_p} \op{Sel}_{p^\infty}(E/\Q)=1$ and $\op{corank}_{\Z_p} \op{Sel}_{p^\infty}(E^{(K)}/\Q)=0$;
        \item $E$ and $E^{(K)}$ have good ordinary reduction at $p$; 
        \item $\widetilde{E}(\F_p)[p]=0$ and $\widetilde{E}^{(K)}(\F_p)[p]=0$;
        \item $p\nmid \prod_{\ell\neq p} c_\ell(E/\Q)\times \prod_{\ell\neq p} c_\ell(E^{(K)}/\Q)$;
        \item the normalized $p$-adic regulator of $E$ over $\Q$ is a $p$-adic unit;
         \item $\Sh(E/\Q)[p^\infty]=0$ and $\Sh(E^{(K)}/\Q)[p^\infty]=0$.
    \end{enumerate}
    There is a line $(a_0, b_0)\in \mathbb{P}^1(\F_p)$ such that for all $(a,b)\not \equiv (a_0, b_0)\pmod{p}$ and for all $n\in \Z_{\geq 0}$, Hilbert's tenth problem has a negative answer for $K_{a,b}^{(n)}$. Moreover, one can take $(a_0, b_0)\neq (1, 0)$. 
\end{lthm}

In many of the cases we consider, the rank of the elliptic curve in question is unbounded in the anticyclotomic $\Z_p$-extension, due to the presence of Heegner points. Thus, we are not in general able to obtain a result that applies to all $\Z_p$-extensions of $K$. For $p=3$ we show unconditionally that there is a large family of imaginary quadratic fields for which an elliptic curve $E_{/\Q}$ satisfies the above conditions; the corresponding assertions for $p=11,13,31,37$ are conditional as specified in Theorem~\ref{lastthm}. It conveniences us to introduce a shorthand for the assertions made in the above mentioned results. 
\begin{definition}
    We say that the pair $(K, p)$ satisfies $(\mylabel{hyp:h10gen}{\textbf{H10-gen}})$ if there exists a line $(a_0,b_0)\in \mathbb{P}^1(\mathbb{F}_p)$ such that $K_{a,b}^{(n)}/\Q$ is integrally diophantine for all $(a,b)$ that are not congruent to $(a_0,b_0)$ modulo $p$. Moreover, $(a_0, b_0)$ can be taken to not be equal to $(1, 0)$. 
\end{definition}
The criterion stands for "Hilbert's tenth problem in generic $\Z_p$-extensions". Endow $\mathbb{P}^1(\Z_p)$ with the natural measure induced from $p$-adic valuation on $\Z_p$. The condition implies that Hilbert's tenth problem is negative for all finite layers in $K_{(a,b)}$ for a density of at least $(1-(p+1)^{-1})$ of all $(a,b)\in \mathbb{P}^1(\Z_p)$). In this section, we fix an odd prime number $p$ and write $K=\Q(\sqrt{d})$, where $d$ is a squarefree negative integer. Set $\op{ht}(K):=|d|$, and consider the following natural question. 
\begin{question}
    Fix an odd prime $p$. When ordered by their height, is there a positive density of imaginary quadratic fields $K$ for which the criterion \eqref{hyp:h10gen} is satisfied?
\end{question}
If the answer to the above question is yes, it shows that there are indeed large families of imaginary quadratic fields for which our results hold, thus giving rise to many new families of number fields for which Hilbert's tenth problem has a negative answer.

\begin{lthm}[Theorems \ref{lastthm1} and \ref{lastthm 2}]\label{lastthm}
\begin{enumerate}
    \item For $p=3$, there is unconditionally a positive density of imaginary quadratic fields for which \eqref{hyp:h10gen} holds.
    \item Let $p\in\{11,13,31,37\}$. Suppose that, for at least one of the triples $(E,K_0,p)$ listed in Example~\ref{important example}, the twists in the associated positive-density rank-zero family for which
    \[
        \Sh(E^{(d)}/\Q)[p^\infty]=0
    \]
    have positive relative lower density in that family. Then there is a positive density of imaginary quadratic fields for which \eqref{hyp:h10gen} holds.
\end{enumerate}
\end{lthm}

\subsection{Organization and methodology} Including the introduction, the article consists of five sections. Section \ref{s 2} is dedicated to establishing preliminary notions and setting up relevant notation. In this section, we discuss the algebraic structure of Selmer groups over infinite towers of number fields. Moreover, we discuss basic notions related to Hilbert's tenth problem and the notion of an integrally diophantine extension. A criterion of Shlapentokh (Theorem \ref{shlap 1}) relates the stability of ranks of elliptic curves to Hilbert's tenth problem. Given an extension of number fields $L/K$, if there is an elliptic curve $E_{/K}$ such that $\op{rank}E(L)=\op{rank}E(K)>0$, then, $L/K$ is an integrally diophantine extension. This is the main tool that we shall use in establishing our results in subsequent sections. Iwasawa theory is used to give explicit conditions for the rank of an elliptic curve to remain constant in a $\Z_p$-extension. In section \ref{s 3}, we discuss the role of the Euler characteristic in Iwasawa theory and its relevance to the study of ranks of elliptic curves. We discuss an explicit formula due to Perrin-Riou and Schneider which is used to calculate Iwasawa invariants of Selmer groups. In section \ref{s 4}, we prove the main results of this article, namely Theorems \ref{main thm 1}, \ref{main thm 2} and \ref{main thm 3}. Given an imaginary quadratic field $K$ and an odd prime $p$, we give explicit conditions for there to exist an elliptic curve $E_{/\Q}$ with positive rank, such that the rank remains constant in many $\Z_p$-extensions $K_{a,b}/K$. This construction relies on understanding the structure of the Selmer group over the $\Z_p^2$-extension of $K$ and relating it to the that over $K_{a,b}$. Explicit examples are provided at the end of this section, cf. Examples \ref{example 4.15} and \ref{example 4.16}. In section \ref{s 5}, we study the conditions of Theorems \ref{main thm 1}, \ref{main thm 2} and \ref{main thm 3} in families of imaginary quadratic fields. The strategy used here is motivated (in part) by the method of Gajek-Leonard et al. \cite{gajek2024conjecture}, who prove Mazur's growth number conjecture for the $\Z_p^2$-extension of an imaginary quadratic field, provided a number of additional conditions are satisfied. Results of Bhargava et al. \cite{BKOS} give the unconditional $p=3$ assertion in Theorem~\ref{lastthm}. The results obtained from Kriz--Li \cite{kriz-li} for $p=11,13,31,37$ are conditional on a positive-relative-density vanishing statement for the relevant $p$-primary Tate--Shafarevich groups.

\subsection{Outlook} Our results illustrate that Iwasawa theory can be employed as a novel tool to establish new cases of the Denef-Lipschitz conjecture. When combined with arithmetic statistics, these techniques give unconditional positive-density families for $p=3$ and conditional positive-density families in the Kriz--Li setting described above. The Iwasawa theory of elliptic curves was introduced by Mazur \cite{mazurgrowth}, whose original motivation was to study the growth of ranks of elliptic curves in $\Z_p$-towers. We consider a broader application of Iwasawa theory which is in keeping with its traditional scope. One may seek to explore more generalized contexts, specifically those involving $p$-adic Lie extensions of dimension $d > 1$. Harris \cite{harris2000correction} provided general upper bounds for the asymptotic growth of the rank of the Mordell-Weil group of an abelian variety in a tower of number fields contained within a $p$-adic Lie extension. These upper bounds are crucial as they give a framework for understanding how the rank behaves as we move up through the layers of the extension. Despite this progress, the precise order of growth in Mordell-Weil ranks remains an open question and is not fully understood, as highlighted in \cite[Question following Proposition 4.7]{hung2021growth}.

\subsection*{Data availability} The article has no associated data. 

\subsection*{Conflict of interest statement} We have no conflict of interest to report.

\section{Preliminaries}\label{s 2}
\par In this section, we examine foundational concepts. In subsection \ref{s 2.1}, we introduce Selmer groups associated with elliptic curves and their connections to the Mordell--Weil group and the Tate-Shafarevich group. These Selmer groups will be analyzed over different $\Z_p$-extensions of an imaginary quadratic field $K$. In subsection \ref{s 2.2}, we address Hilbert's tenth problem and the relative notion of integrally diophantine extensions of number fields.

\subsection{Selmer groups and their Iwasawa theory}\label{s 2.1}
Throughout $d$ will be a squarefree negative integer, $K := \mathbb{Q}(\sqrt{d})$ and let $ p $ be an odd prime number. Set $ K_\infty/K$ to denote the $\mathbb{Z}_p^2$-extension of $ K $. The \textit{cyclotomic $\mathbb{Z}_p$-extension} $ K_{\mathrm{cyc}}/K $ is the unique $\mathbb{Z}_p$-extension of $K$ contained within $ K(\mu_{p^\infty}) $. On the other hand, the \textit{anticyclotomic $\mathbb{Z}_p$-extension} $ K_{\mathrm{ac}}/K $ is a Galois extension over $\mathbb{Q}$, where the induced action of $\mathrm{Gal}(K/\mathbb{Q})$ on $\mathrm{Gal}(K_{\mathrm{ac}}/K)$ is nontrivial. We note that $ K_{\mathrm{cyc}} $ and $ K_{\mathrm{ac}} $ are both subfields of $ K_\infty $ and that $ K_{\infty}=K_{\mathrm{cyc}} \cdot K_{\mathrm{ac}} $. Set $\Gamma_\infty := \mathrm{Gal}(K_\infty/K)$, $\Gamma_{\mathrm{cyc}} := \mathrm{Gal}(K_{\mathrm{cyc}}/K)$, and $\Gamma_{\mathrm{ac}} := \mathrm{Gal}(K_{\mathrm{ac}}/K)$.

\par Given a pro-$p$ group $\mathcal{G}$, the \textit{Iwasawa algebra} is defined as the inverse limit
\[
\Lambda(\mathcal{G}) = \varprojlim_U \mathbb{Z}_p\left[\mathcal{G}/U\right],
\]
where $U$ runs over all finite index normal subgroups of $\mathcal{G}$. The Iwasawa algebras associated with $\Gamma_\infty$, $\Gamma_{\mathrm{cyc}}$, and $\Gamma_{\mathrm{ac}}$ are denoted by $\Lambda_\infty$, $\Lambda_{\mathrm{cyc}}$, and $\Lambda_{\mathrm{ac}}$, respectively. Note that $\Gamma_{\op{cyc}}$ and $\Gamma_{\op{ac}}$ are both quotients of $\Gamma_\infty$, and that the natural quotient maps induce surjections at the level of Iwasawa algebras:
\[\pi_{\op{cyc}}: \Lambda_\infty\rightarrow \Lambda_{\op{cyc}}\text{ and }\pi_{\op{ac}}: \Lambda_\infty\rightarrow \Lambda_{\op{ac}}.\]
\par The Iwasawa theory of elliptic curves concerns itself with the structure of Selmer groups, viewed over infinite Galois extensions. These Selmer groups naturally give rise to modules over Iwasawa algebras, and it is the structure theory of such modules that give rise to Iwasawa theoretic arithmetic invariants. The elliptic curves $E_{/\Q}$ we shall focus on in this article will be defined over $\Q$ and will have good ordinary reduction at $p$.
\par Let $E_{/\mathbb{Q}} $ be an elliptic curve and denote by $ E^{(K)} $ the quadratic twist of $ E $ by $ K $. Assume that both $E$ and $E^{(K)}$ have good ordinary reduction at $p$. Let $ N_E $ denote the conductor of $ E $ and let $\Sigma$ be a finite set of rational primes containing the primes that divide $ N_E p $. Denote by $\Q_\Sigma$ the maximal algebraic extension of $\Q$ in which the finite primes outside $\Sigma$ are unramified. We introduce the $p$-primary Selmer group of $E$ over an algebraic extension of $\Q$. For any number field $ F $, let $\Sigma(F)$ be the set of primes $ v $ of $ F $ that lie above $\Sigma$. Assume that $F$ is contained in $\Q_\Sigma$, i.e., $\Sigma$ contains the primes that ramify in $F$. For $\ell \in \Sigma$, set
\[
J_\ell(E/F) := \bigoplus_{v|\ell} H^1(F_v, E)[p^\infty].
\]
Let $ \mathcal{F}/F$ be an infinite pro-$p$ extension of $F$, denote by \[J_\ell(E/\mathcal{F}):=\varinjlim_{F'} J_\ell(E/F')\]where $ F'/F$ ranges over number field extensions of $F$ contained in $\mathcal{F}$. The $ p $-primary Selmer group of $ E $ over $ \mathcal{F} $ is then defined as follows:
\[
\mathrm{Sel}_{p^\infty} (E/\mathcal{F}) := \ker\left\{H^1(\Q_\Sigma/\mathcal{F}, E[p^\infty]) \longrightarrow \bigoplus_{\ell \in \Sigma} J_\ell(E/\mathcal{F})\right\}.
\]
We take note of a useful statement about Selmer groups and change of base.

\begin{proposition}\label{prop:selmer-decomposition-ord}
    Suppose $ E/\mathbb{Q} $ is an elliptic curve, and let $ E^{(K)} $ be the twist of $ E $ corresponding to $ K $. Then,
    \[
    \mathrm{Sel}_{p^\infty}(E/K_{\mathrm{cyc}}) \simeq \mathrm{Sel}_{p^\infty}(E/\mathbb{Q}_{\mathrm{cyc}}) \oplus \mathrm{Sel}_{p^\infty}(E^{(K)}/\mathbb{Q}_{\mathrm{cyc}}).
    \]
\end{proposition}
\begin{proof}
    From Shapiro's lemma we deduce that 
    \[H^1(\Q_\Sigma/K_{\op{cyc}}, E[p^\infty])\simeq H^1(\Q_\Sigma/\Q_{\op{cyc}}, E[p^\infty])\oplus H^1(\Q_\Sigma/\Q_{\op{cyc}}, E^{(K)}[p^\infty]),\] and 
    \[J_\ell(E/K_{\op{cyc}})=J_\ell(E/\Q_{\op{cyc}})\oplus J_\ell(E^{(K)}/\Q_{\op{cyc}}).\]
    These isomorphisms are compatible with the restriction maps, and we thus find that 
    \[
    \mathrm{Sel}_{p^\infty}(E/K_{\mathrm{cyc}}) \simeq \mathrm{Sel}_{p^\infty}(E/\mathbb{Q}_{\mathrm{cyc}}) \oplus \mathrm{Sel}_{p^\infty}(E^{(K)}/\mathbb{Q}_{\mathrm{cyc}}).
    \]
\end{proof}

Let $ \sigma, \tau \in \Gamma_\infty $ be topological generators satisfying
\[
\begin{split}
    &\overline{\langle \sigma \rangle} = \ker\left\{\Gamma_\infty \longrightarrow \Gamma_{\mathrm{cyc}}\right\},\\
    & \overline{\langle \tau \rangle} = \ker\left\{\Gamma_\infty \longrightarrow  \Gamma_{\mathrm{ac}}\right\}.
\end{split}
\]
Denote $ X+1 $ as $ \sigma $ and $ Y+1 $ as $ \tau $. Then we can identify $ \Lambda_\infty$ with the ring of power series $\mathbb{Z}_p\llbracket X,Y\rrbracket $. Let $\mathbb{P}^1(\mathbb{Z}_p)$ be the space of lines in $\mathbb{Z}_p^2$. Any line has a representative $ (a,b) \in \mathbb{Z}_p^2 $ such that $ p \nmid a $ or $ p \nmid b $, which is uniquely determined up to multiplication by a unit in $\Z_p$. Given $ (a,b) \in \mathbb{P}^1(\mathbb{Z}_p) $, there is a unique $\mathbb{Z}_p$-extension $ K_{a,b}/K $ such that
\[
\overline{\langle \sigma^a \tau^b \rangle} = \ker\left(\Gamma_\infty \rightarrow \mathrm{Gal}(K_{a,b}/K)\right).
\]
For $ (a,b) \in \mathbb{P}^1(\mathbb{Z}_p) $, set $\Gamma_{a,b} = \mathrm{Gal}(K_{a,b}/K)$, $ H_{a,b} = \overline{\langle \sigma^a \tau^b \rangle} $, and $\Lambda_{a,b} = \mathbb{Z}_p\llbracket \Gamma_{a,b} \rrbracket $. Define $\pi_{a,b} : \Lambda_\infty \rightarrow \Lambda_{a,b}$ to be the map induced by the natural projection $ \Gamma_\infty \rightarrow \Gamma_{a,b} $. Set $\omega_{a,b}:= (1+X)^a(1+Y)^b - 1 $. For $(a,b) = (1:0)$, use $\Gamma_{\mathrm{cyc}}$, $ H_{\mathrm{cyc}}$, and $\pi_{\mathrm{cyc}}$ instead of $\Gamma_{(1:0)}$, $ H_{(1:0)}$, and $\pi_{(1:0)}$, respectively. Choose a topological generator $\gamma=\gamma_{a,b}$ of $\Gamma_{a,b}$ and set $T=T_{a,b}:=\gamma-1$. Then, we identify $\Lambda_{a,b}$ with the formal power series ring $\Z_p\llbracket T\rrbracket$. In particular, identify $\Lambda_{\rm{cyc}}$ (resp. $\Lambda_{\rm{ac}}$) with $\Z_p\llbracket T\rrbracket$.

\par The Pontryagin dual of a module $M$ over $\Z_p\llbracket T\rrbracket$ is defined by \[M^\vee:=\op{Hom}_{\Z_p}\left(M, \Q_p/\Z_p\right).\] We say that $M$ is \emph{cofinitely generated} (resp. \emph{cotorsion}) over $\Z_p\llbracket T\rrbracket$ if $M^\vee$ is finitely generated (resp. torsion). A monic polynomial $ f(T) \in \mathbb{Z}_p\llbracket T \rrbracket $ is called a \textit{distinguished polynomial} if all its coefficients, except for the leading one, are divisible by $ p $. Let $ M $ and $ M' $ be cofinitely generated and cotorsion $\mathbb{Z}_p\llbracket T\rrbracket$-modules. Then $ M $ and $ M' $ are \textit{pseudo-isomorphic} if there is a $\mathbb{Z}_p\llbracket T \rrbracket $-module map $\phi: M \rightarrow M'$ whose kernel and cokernel are finite.

\par According to the structure theory of $\Z_p\llbracket T\rrbracket$-modules, any cofinitely generated and cotorsion $\Z_p\llbracket T\rrbracket$-module $ M $ is pseudo-isomorphic to a module $ M' $ whose Pontryagin dual is a direct sum of cyclic torsion $\Z_p\llbracket T\rrbracket$-modules:
\begin{equation}\label{structure isomorphism}
(M')^\vee \simeq \left( \bigoplus_{i=1}^s \frac{\Z_p\llbracket T\rrbracket}{(p^{n_i})} \right) \oplus \left( \bigoplus_{j=1}^t \frac{\Z_p\llbracket T\rrbracket}{(f_j(T))} \right).
\end{equation}
Here, $ s $ and $ t $ are natural numbers (possibly $ 0 $), $ n_i \in \mathbb{Z}_{\geq 1} $, and $ f_j(T) $ are distinguished polynomials. The Iwasawa $\mu$- and $\lambda$-invariants are then defined as follows:
\[
\mu_p(M) := \sum_{i=1}^s n_i \quad \text{and} \quad \lambda_p(M) := \sum_{j=1}^t \deg f_j.
\]
If $ s = 0 $ (respectively, $ t = 0 $), the sum $\sum_{i=1}^s n_i$ (respectively, $\sum_{j=1}^t \deg f_j$) is taken to be $ 0 $. 
\begin{lemma}\label{basic lemma on lambda and corank}
    Let $M$ be a cofinitely generated and cotorsion $\Z_p\llbracket T\rrbracket$-module. Then the following conditions are equivalent. 
    \begin{enumerate}
        \item The module $M$ is cotorsion $\mathbb{Z}_p\llbracket T\rrbracket $-module and $\mu_p(M)=0$. 
        \item The module $M$ is cofinitely generated as a $\Z_p$-module. 
    \end{enumerate}
    Moreover, if the above (equivalent) conditions are satisfied, then, \[M\simeq \left(\Q_p/\Z_p\right)^{\lambda_p(M)}\oplus A,\] where $A$ is a finite $\Z_p$-module. 
\end{lemma}
\begin{proof}
    The result is an easy consequence of the structure isomorphism \eqref{structure isomorphism}, the details are omitted. 
\end{proof}

\par It then follows from results of Kato \cite{Katomodform} and Rubin \cite{RubinBSD} that $\op{Sel}_{p^\infty}(E/K_{\op{cyc}})$ is cofinitely generated and cotorsion over $\Lambda_{\op{cyc}}$. 

\subsection{Hilbert's tenth problem for number rings}\label{s 2.2}
Let $ A $ be a commutative ring with an identity element, and consider $ A^n $, which is the free $ A $-module of rank $ n $. This module consists of tuples $ a = (a_1, \dots, a_n) $ where each $ a_i $ is an element of $ A $. Suppose $ m $ and $ n $ are positive integers. For $ a = (a_1, \dots, a_n) \in A^n $ and $ b = (b_1, \dots, b_m) \in A^m $, we denote by $ (a,b) \in A^{n+m} $ the concatenated tuple $ (a_1, \dots, a_n, b_1, \dots, b_m) $. Given a finite collection of polynomials $ F_1, \dots, F_k $, we define the set 
\[
\mathcal{F}(a; F_1, \dots, F_k) := \{ b \in A^m \mid F_i(a,b) = 0 \text{ for all } i = 1, \dots, k \}.
\]

\begin{definition}\label{diophantine subset}
A subset $ S $ of $ A^n $ is called a \emph{diophantine subset} of $ A^n $ if there exists an integer $ m \geq 1 $ and a collection of polynomials \[F_1, \dots, F_k \in A[x_1, \dots, x_n, y_1, \dots, y_m] \] such that $ S $ consists of all $ a \in A^n $ for which the set $ \mathcal{F}(a; F_1, \dots, F_k) $ is nonempty.
\end{definition} 

\begin{definition}
An extension of number fields $ L/K $ is said to be \emph{integrally diophantine} if the ring of integers $ \mathcal{O}_K $ is a diophantine subset of the ring of integers $ \mathcal{O}_L $.
\end{definition} Let $ L' $, $ L $, and $ K $ be number fields such that 
\[ L' \supseteq L \supseteq K. \]
We assume that both $ L'/L $ and $ L/K $ are integrally diophantine extensions. It is a basic fact in the theory of diophantine sets and fields that, under these conditions, the composite extension $ L'/K $ is also integrally diophantine. This assertion follows from \cite[Theorem 2.1.15]{shlapentokh2007hilbert}. In fact, \cite{shlapentokh2007hilbert} provides a comprehensive treatment of Hilbert's tenth problem for rings of integers in number fields.

In the context of integrally diophantine extensions, we recall a conjecture proposed by Denef and Lipshitz \cite{denef1978diophantine}.

\begin{conjecture}[Denef-Lipshitz]\label{denef lipchitz conjecture}
For any number field $ L $, the extension $ L/\Q $ is integrally diophantine.
\end{conjecture}

This conjecture posits that for any number field $ L $, the ring of integers $\cO_\Q = \Z$ is a diophantine subset of the ring of integers $\cO_L$. The veracity of this conjecture would imply a significant expansion of the negative resolution of Hilbert's tenth problem. 

The renowned result by Matijasevich, building on work by Davis, Putnam, and Robinson, demonstrated that Hilbert's tenth problem, which asks whether there is an algorithm to determine the solvability of arbitrary diophantine equations over $\Z$, has a negative answer. Specifically, there is no such algorithm. If Conjecture \ref{denef lipchitz conjecture} holds for a number field $ L $, it would imply that Hilbert's tenth problem also has a negative answer for the ring of integers $\cO_L$. In other words, there would be no algorithm to determine the solvability of diophantine equations over $\cO_L$. In addition to the cases discussed in the introduction, there have been significant advancements in establishing new criteria for the validity of Conjecture \ref{denef lipchitz conjecture}. Notable contributions have been made by Poonen \cite{poonen2002using}, Cornelissen, Pheidas, and Zahidi \cite{cornelissen2005division}, and Shlapentokh \cite{shlapentokh2008elliptic}. These researchers have introduced criteria linked to the stability of the Mordell-Weil rank of elliptic curves over number fields.

One particularly influential result is Shlapentokh's criterion, which relates diophantine stability for elliptic curves with positive rank to Hilbert's tenth problem for rings of integers in number fields:

\begin{theorem}[Shlapentokh]\label{shlap 1}
Let $ L/K $ be an extension of number fields, and suppose there exists an elliptic curve $ E $ defined over $ K $ such that $\operatorname{rank} E(L) = \operatorname{rank} E(K) > 0$. Then, $ L/K $ is integrally diophantine. Consequently, if Hilbert's tenth problem has a negative answer for $\cO_K$, then Hilbert's tenth problem also has a negative answer for $\cO_L$.
\end{theorem}

This theorem demonstrates that the existence of an elliptic curve $ E $ over $ K $ with equal positive rank over both $ K $ and $ L $ ensures that the ring of integers $\cO_K$ is a diophantine subset of $\cO_L$. This result provides a powerful tool for transferring the negative solution of Hilbert's tenth problem from $\cO_K$ to $\cO_L$, thereby expanding the range of number fields for which the problem is known to have no algorithmic solution.

\section{Control theory for Selmer groups and the Euler characteristic formula}\label{s 3}

\par In this section, we briefly review Mazur's control theorem as well as the Euler characteristic formula for the Selmer group, due to Perrin-Riou and Schneider. Throughout we let $p$ be an odd prime number, $E$ be an elliptic curve over a number field $F$ and let $F_{\op{cyc}}$ be the cyclotomic $\Z_p$-extension of $F$. Fix an algebraic closure $\bar{F}$ of $F$. Set $F_{\op{cyc}}^{(n)}/F$ to be the subextension for which $[F_{\op{cyc}}^{(n)}:F]=p^n$. Let $\Gamma_{F_{\op{cyc}}}:=\op{Gal}(F_{\op{cyc}}/F)$ and $\Gamma_{F_{\op{cyc}}, n}:=\op{Gal}(F_{\op{cyc}}/F_n)=\Gamma_{F_{\op{cyc}}}^{p^n}$. Let $S$ be a set of primes of $F$ containing the primes $v|p$ and the primes at which $E$ has bad reduction. Denote by $F_S$ the maximal extension of $F$ that is contained in $\bar{F}$, in which the primes $v\notin S$ are unramified. 

\subsection{Mazur's control theorem}
There is a natural map
\[
\alpha_n: \op{Sel}_{p^\infty}(E/F_{\op{cyc}}^{(n)}) \longrightarrow \op{Sel}_{p^\infty}(E/F_{\op{cyc}})^{\Gamma_{F_{\op{cyc}}, n}},
\]
which fits into the following commutative diagram:
\[
\begin{tikzcd}[column sep = small, row sep = large]
0 \arrow{r} & \op{Sel}_{p^\infty}(E/F_{\op{cyc}}^{(n)}) \arrow{r} \arrow{d}{\alpha_n} & H^1(F_S/F_{\op{cyc}}^{(n)}, E[p^\infty]) \arrow{r} \arrow{d}{\beta_n} & \bigoplus_{v \in S} J_v(E/F_{\op{cyc}}^{(n)}) \arrow{d}{\gamma_n} \\
0 \arrow{r} & \op{Sel}_{p^\infty}(E/F_{\op{cyc}})^{\Gamma_{F_{\op{cyc}}, n}} \arrow{r} & H^1(F_S/F_{\op{cyc}}, E[p^\infty])^{\Gamma_{F_{\op{cyc}}, n}} \arrow{r} & \left(\bigoplus_{v \in S} J_v(E/F_{\op{cyc}})\right)^{\Gamma_{F_{\op{cyc}}, n}}.
\end{tikzcd}
\]
Here, the maps $\beta_n$ and $\gamma_n$ are induced by the restriction maps in Galois cohomology.

\begin{theorem}[Mazur]
    The maps $\alpha_n$ described above have finite kernels and cokernels whose orders are uniformly bounded, independent of $n$.
\end{theorem}

We now present a criterion for $\op{Sel}_{p^\infty}(E/F_{\op{cyc}})$ to be a cofinitely generated and cotorsion $\Z_p\llbracket T\rrbracket$-module.

\begin{theorem}
    If $\op{Sel}_{p^\infty}(E/F)$ is finite, then $\op{Sel}_{p^\infty}(E/F_{\op{cyc}})$ is cofinitely generated and cotorsion.
\end{theorem}

\begin{proof}
    For the proof, see \cite[Theorem 1.4]{GreenbergIwasawa}.
\end{proof}

\begin{theorem}[Kato–Rohrlich]
    Assume $E$ is defined over $\mathbb{Q}$ and $F$ is an abelian extension of $\mathbb{Q}$. Then, $\op{Sel}_{p^\infty}(E/F_{\op{cyc}})$ is cofinitely generated and cotorsion.
\end{theorem}
\subsection{Generalities on the Euler characteristic}
In this section, let $\Gamma\simeq \Z_p$ and $\Lambda:=\Z_p\llbracket \Gamma \rrbracket$. Let $ M $ be a $\Lambda$-module that is cofinitely generated and cotorsion. We focus on two derived modules: the module of invariants $ H^0(\Gamma, M) = M^\Gamma $ and the module of co-invariants $ H^1(\Gamma, M) = M_\Gamma = M / TM $. There is a natural map between these modules:
\[
\phi_M: M^\Gamma \rightarrow M_\Gamma,
\]
sending an element $ x \in M^\Gamma $ to $ x \mod{TM} $ in $ M_\Gamma $. Since $ M $ is cofinitely generated over $\Lambda$, both $ M^\Gamma $ and $ M_\Gamma $ are cofinitely generated over $\Z_p$. Additionally, since $\Gamma \cong \Z_p$ has cohomological dimension 1, $ H^i(\Gamma, \cdot) = 0 $ for $ i \geq 2 $. For such a module $ M $, the coranks of $ M^\Gamma $ and $ M_\Gamma $ as $\Z_p$-modules are equal:
    \[
    \operatorname{corank}_{\Z_p} M^\Gamma = \operatorname{corank}_{\Z_p} M_\Gamma,
    \]
cf. \cite[Lemma 2.1]{stats1}. A key consequence of this lemma is that $ M^\Gamma $ is finite if and only if $ M_\Gamma $ is finite.

\begin{definition}
    Let $ M $ be a $\Lambda$-module that is cofinitely generated and cotorsion. The \emph{Euler characteristic} of $ M $ is well-defined if $ M^\Gamma $ (and equivalently $ M_\Gamma $) is finite. In this case, the \emph{Euler characteristic} of $ M $ is defined as:
    \[
    \chi(\Gamma, M) := \prod_{i \geq 0} \left(H^i(\Gamma, M)\right)^{(-1)^i} = \left(\frac{\# M^\Gamma}{\# M_\Gamma}\right).
    \]
\end{definition}

There is indeed a generalization of the above notion when $M^\Gamma$ is infinite which we shall treat in the next subsection.
We define the characteristic series of $ M $ as follows.

\begin{definition}
Consider the numbers $n_i$ and polynomials $f_j$ in \eqref{structure isomorphism}. Then the characteristic series $ f_M(T) $ is defined by:
    \[
    f_M(T) := \prod_i p^{n_i} \times \prod_j f_j(T),
    \]
    and can be expressed as:
    \[
    f_M(T) = a_0 + a_1 T + a_2 T^2 + \dots + a_\lambda T^\lambda.
    \]
\end{definition}

\begin{proposition}\label{Prop 3.6}
    For a $\Lambda$-module $ M $ that is cofinitely generated and cotorsion, the following conditions are equivalent:
    \begin{enumerate}
        \item The Euler characteristic $\chi(\Gamma, M)$ is well-defined.
        \item The constant term $ a_0 $ of the characteristic series $ f_M(T) $ is non-zero.
    \end{enumerate}
    Furthermore, if these conditions are satisfied, then $\chi(\Gamma, M)$ is an integer, and it holds that:
    \[
    a_0 \sim \chi(\Gamma, M),
    \]
    where $ a \sim b $ means $ a = ub $ for some unit $ u \in \Z_p^\times $.
\end{proposition}

\begin{proof}
   For a proof of the result, see \cite[Proposition 4.4]{ray2024analogue}.
\end{proof}

\begin{lemma}\label{lem:classical-euler-characteristic}
    For a cofinitely generated and cotorsion $\Lambda$-module $ M $ with a well-defined Euler characteristic, the following conditions are equivalent:
    \begin{enumerate}
        \item\label{p1} $\mu_p(M) = 0$ and $\lambda_p(M) = 0$,
        \item\label{p2} $ a_0 $ is not divisible by $ p $,
        \item\label{p3} $\chi(\Gamma, M) = 1$.
    \end{enumerate}
\end{lemma}

\begin{proof}
    This result is \cite[Lemma 4.5]{ray2024analogue}.
\end{proof}

\subsection{The truncated Euler characteristic}

\par When the cohomology groups $ H^i(\Gamma, M) $ are not finite, a generalized version called the truncated Euler characteristic, denoted $ \chi_t(\Gamma, M) $, is used. Given that $ H^1(\Gamma, M) $ is isomorphic to the group of coinvariants $ H_0(\Gamma, M) = M_{\Gamma} $, there exists a natural map 
\[
\Phi_{M}: M^{\Gamma} \rightarrow M_{\Gamma}
\]
which sends $ x \in M^{\Gamma} $ to its residue class in $ M_{\Gamma} $. The truncated Euler characteristic $ \chi_t(\Gamma, M) $ is defined when both the kernel and the cokernel of $ \Phi_{M} $ are finite. It is given by the formula
\[
\chi_t(\Gamma, M) := \frac{\# \operatorname{ker}(\Phi_{M})}{\# \operatorname{cok}(\Phi_{M})}.
\]
It can be shown that if the classical Euler characteristic $ \chi(\Gamma, M) $ is defined, then $ \chi_t(\Gamma, M) $ is also defined, and in fact,
\[
\chi_t(\Gamma, M) = \chi(\Gamma, M).
\]

Let $ r_{M} $ denote the order of vanishing of $ f_M(T) $ at $ T = 0 $. For $ a, b \in \mathbb{Q}_p $, recall that $ a \sim b $ if there exists a unit $ u \in \mathbb{Z}_p^{\times} $ such that $ a = bu $.

\begin{lemma}
\label{lemmazerbes1}
Let $ M $ be a cofinitely generated cotorsion $ \Z_p\llbracket T\rrbracket $-module. Assume that the kernel and cokernel of $ \Phi_{M} $ are finite. Then,
\begin{enumerate}
\item $ r_{M} = \operatorname{corank}_{\mathbb{Z}_p}(M^{\Gamma}) = \operatorname{corank}_{\mathbb{Z}_p}(M_{\Gamma}) $.
\item $ a_{r_{M}} \neq 0 $.
\item $ a_{r_{M}} \sim \chi_t(\Gamma, M) $.
\end{enumerate}
\end{lemma}

\begin{proof}
    For the proof, see \cite[Lemma 2.11]{zerbes09}.
\end{proof}

In particular, the classical Euler characteristic $ \chi(\Gamma, M) $ is defined if and only if $ r_{M} = 0 $. Next, we give a criterion for the truncated Euler characteristic to be defined. 

\begin{lemma}\label{lemma zerbes2}
    Let $M$ be a cofinitely generated and cotorsion $\Z_p\llbracket T\rrbracket$-module. Let $f_1(T), \dots, f_t(T)$ be distinguished polynomials in \eqref{structure isomorphism}. Assume that $T^2\nmid f_j$ for all $j=1, \dots, t$. Then the truncated Euler characteristic $\chi_t(\Gamma, M)$ is well defined. In particular, $\chi_t(\Gamma, M)$ is defined when $r_M\leq 1$. 
\end{lemma}
\begin{proof}
    The assertion follows from the proof of \cite[Lemma 2.11]{zerbes09}.
\end{proof}

\begin{lemma}\label{TEC is well defined}
    Let $M$ be a cofinitely generated and cotorsion $\Z_p\llbracket T\rrbracket$-module for which $\op{corank}_{\Z_p}(M^\Gamma)= 1$. Then, one has that $r_M=1$. Moreover, the truncated Euler characteristic is defined and 
    \[\chi_t(\Gamma, M)\sim a_1.\]
\end{lemma}
\begin{proof}
    The result follows from Lemmas \ref{lemmazerbes1} and \ref{lemma zerbes2}.
\end{proof}

\subsection{A formula for the truncated Euler characteristic}
\par Let $E$ be an elliptic curve over a number field $F$. Assume that $E$ has good ordinary reduction at the primes of $F$ that lie above $p$. We set $\Gamma:=\op{Gal}(F_{\op{cyc}}/F)$. Assume that the Selmer group $\op{Sel}_{p^\infty}(E/F_{\op{cyc}})$ is cofinitely generated and cotorsion as a $\Lambda_{\op{cyc}}$-module. Note that this is indeed true when $F=\Q$, or if $F$ is an abelian extension of $\Q$ and $E$ is defined over $\Q$. Assume that the $\Z_p$-corank of $\op{Sel}_{p^\infty}(E/F)$ is $1$. Consider the natural map
\[\Phi: \op{Sel}_{p^\infty} (E/F_{\op{cyc}})^\Gamma\rightarrow \op{Sel}_{p^\infty} (E/F_{\op{cyc}})_{\Gamma}\]
sending an element $x\in \op{Sel}_{p^\infty} (E/F_{\op{cyc}})^\Gamma$ to its residue class $\Phi(x)=\bar{x}\in \op{Sel}_{p^\infty} (E/F_{\op{cyc}})_{\Gamma}$. It follows from Mazur's control theorem that $\op{corank}_{\Z_p}\op{Sel}_{p^\infty}(E/F_{\op{cyc}})^{\Gamma_{\op{cyc}}}=1$.
The truncated Euler characteristic $\chi_t(\Gamma, E[p^\infty])$ is defined as the following quotient
\[\chi_t(\Gamma, E[p^\infty]):=\frac{\# \op{ker}\Phi}{\#\op{cok} \Phi}.\] 
Note that by Lemma \ref{TEC is well defined}, $\op{ker}\Phi$ and $\op{cok}\Phi$ are finite, and thus the truncated Euler characteristic is well defined.

\par Let $f(T)$ represent the characteristic element of $\operatorname{Sel}_{p^\infty}(E/F_{\operatorname{cyc}})$. Let $r\in \mathbb{Z}_{\geq 0}$ be the order of vanishing of $f(T)$ at $T=0$, and express $f(T)$ as $T^r \cdot g(T)$. According to Lemma \ref{TEC is well defined}, $r$ is equal to 1. Note that $g(0) \neq 0$. The $p$-adic regulator of $E$ over $F$, denoted as $\operatorname{Reg}_p(E/F)$, is defined as the determinant of the $p$-adic height pairing (see \cite{schneiderhp1}). It is conjectured that the $p$-adic regulator is non-zero (see \cite{Schneider85}). If $r_E:=\operatorname{rank}E(F)$, let $\mathcal{R}_p(E/F):=p^{-r_E}\operatorname{Reg}_p(E/F)$ denote the normalized regulator. Thus $\mathcal{R}_p(E/F)=\operatorname{Reg}_p(E/F)/p$ in the rank-one situation considered above, while in rank zero we use the usual empty-determinant convention $\operatorname{Reg}_p(E/F)=\mathcal{R}_p(E/F)=1$. Given a prime $v$ of $K$, denote by $c_v(E/F)$ the Tamagawa number of $E$ over $F_v$. Set $\kappa_v$ to denote the residue field at $v$. Suppose $E$ has good reduction at $v$, set $\widetilde{E}(\kappa_v)$ to be the group of $\kappa_v$-rational points of the reduction of a regular model of $E$ at the prime $v$. 
\par The following result gives the formula for the truncated Euler characteristic of the $p$-primary Selmer group when it is defined.
 In the CM case, this was proven by B. Perrin-Riou (see \cite{PR82}) and in the general case by P. Schneider (see \cite{Schneider85}).

\begin{theorem}[Perrin--Riou and Schneider]\label{th ECF}
Let $E/F$ be an elliptic curve over a number field $F$ and let $p$ be an odd prime. Assume that the following conditions are satisfied:
\begin{enumerate}
    \item $\Sel_{p^\infty}(E/F_{\op{cyc}})$ is cofinitely generated and cotorsion over $\Lambda_{\op{cyc}}$, and $\operatorname{rank}E(F)\in\{0,1\}$;
    \item $E(F)[p]=0$;
    \item the $p$-primary part of the Tate--Shafarevich group $\Sh(E/F)[p^\infty]$ is finite;
    \item the $p$-adic height pairing on $E(F)$ is nondegenerate.
\end{enumerate}
Then
\[
\chi_t\bigl(\Gamma,\Sel_{p^\infty}(E/F_{\mathrm{cyc}})\bigr)
\sim
\mathcal{R}_p(E/F)\,
\#\Sh(E/F)[p^\infty]\,
\prod_{v\nmid p} c_v(E/F)
\left(
    \prod_{v\mid p}
    \#\widetilde{E}(\kappa_v)[p^\infty]
\right)^2.
\]
Here $\sim$ denotes equality up to multiplication by an element of $\mathbb{Z}_p^\times$. When $\operatorname{rank}E(F)=0$, the truncated Euler characteristic in this formula is the classical Euler characteristic. When $\operatorname{rank}E(F)=1$, it is the rank-one truncated Euler characteristic defined above.
\end{theorem}
 The following result gives a relationship between the Iwasawa invariants and the (truncated) Euler characteristic.
\begin{corollary}\label{cor ECF mulambda}
Let $E/F$ be an elliptic curve satisfying the conditions of the above theorem. Assume in addition that $\op{corank}_{\Z_p}\Sel_{p^\infty}(E/F)=1$. Then with respect to notation above, the following conditions are equivalent.
\begin{enumerate}
    \item The Iwasawa invariants of $\op{Sel}_{p^\infty}(E/F_{\op{cyc}})$ are given by \[\mu\left(\op{Sel}_{p^\infty}(E/F_{\op{cyc}})\right)=0\text{ and }\lambda\left(\op{Sel}_{p^\infty}(E/F_{\op{cyc}})\right)=1;\]
    \item $\chi_t(\Gamma, E[p^\infty])=1$;
    \item the following are all satisfied:
    \begin{itemize}
        \item $p\nmid \mathcal{R}_p(E/F)$,
        \item $\# \Sh(E/F)[p^\infty]=1$,
        \item $p\nmid c_v(E/F)$ for all primes $v\nmid p$ of $F$,
        \item and $\widetilde{E}(\kappa_v)[p^\infty]=0$ for all primes $v|p$ of $F$.
    \end{itemize}   
\end{enumerate}
\end{corollary}
\begin{proof}
Recall that $f(T)=T g(T)$. Thus, $\mu=0$ and $\lambda=1$ if and only if $g(T)$ is a unit in $\Z_p\llbracket T\rrbracket$. On the other hand, $\chi_t(\Gamma, E[p^\infty])\sim g(0)$ (cf. Lemma \ref{lemmazerbes1}, part (3)). Thus, $\mu=0$ and $\lambda=1$ if and only if $\chi_t(\Gamma, E[p^\infty])=1$. This proves that (1) and (2) are equivalent. The equivalence of (2) and (3) follows from Theorem \ref{th ECF}.
\end{proof}

\begin{corollary}[Rank-zero criterion]\label{cor:rank-zero}
Let $E/F$ satisfy the hypotheses of Theorem~\ref{th ECF}, and assume that $\op{corank}_{\Z_p}\Sel_{p^\infty}(E/F)=0$. Then the following conditions are equivalent:
\begin{enumerate}
    \item
    \[
    \mu\bigl(\Sel_{p^\infty}(E/F_{\op{cyc}})\bigr)=0
    \quad\text{and}\quad
    \lambda\bigl(\Sel_{p^\infty}(E/F_{\op{cyc}})\bigr)=0;
    \]
    \item
    \[
    \chi\bigl(\Gamma,\Sel_{p^\infty}(E/F_{\op{cyc}})\bigr)=1;
    \]
    \item $\Sh(E/F)[p^\infty]=0$, $p\nmid c_v(E/F)$ for every $v\nmid p$, and $\widetilde E(\kappa_v)[p^\infty]=0$ for every $v\mid p$.
\end{enumerate}
\end{corollary}
\begin{proof}
Since $\Sh(E/F)[p^\infty]$ is finite, the Selmer exact sequence shows that $E(F)$ has rank zero. Mazur's control theorem therefore makes the $\Gamma$-invariants of $\Sel_{p^\infty}(E/F_{\op{cyc}})$ finite, so its classical Euler characteristic is defined. The equivalence of (1) and (2) is Lemma~\ref{lem:classical-euler-characteristic}. Theorem~\ref{th ECF}, together with $\mathcal{R}_p(E/F)=1$ in rank zero, gives the equivalence of (2) and (3).
\end{proof}

\section{Hilbert's tenth problem in towers of number fields}\label{s 4}

\par In this section, we apply ideas from Iwasawa theory to prove Hilbert's tenth problem in the finite layers of various $\Z_p$-extensions of interest. Our results require that there exist elliptic curves for which certain additional conditions are satisfied. We shall present examples to illustrate our results. The main results of this section are Theorems \ref{main thm 1}, \ref{main thm 2} and \ref{main thm 3}. In the next section, we verify these conditions unconditionally for a positive-density family when $p=3$. For the families obtained from Kriz--Li, the required vanishing of the $p$-primary Tate--Shafarevich group of the rank-zero twists is retained as an explicit density hypothesis.

\par As in the previous section, fix an elliptic curve $E_{/\Q}$ and an odd prime number $p$. We shall let $K=K_d=\Q(\sqrt{d})$ be an imaginary quadratic field. Here, $d$ is a negative squarefree integer. Denote the twist of $E$ over $K$ by $E^{(K)}$, or equivalently by $E^{(d)}$. Recall that Proposition \ref{prop:selmer-decomposition-ord} asserts that $p$-primary Selmer group over $K_{\op{cyc}}$ decomposes as
\[\op{Sel}_{p^\infty}(E/K_{\op{cyc}})=\op{Sel}_{p^\infty}(E/\Q_{\op{cyc}})\oplus \op{Sel}_{p^\infty}(E^{(K)}/\Q_{\op{cyc}}).\]

\begin{ass}
    We shall assume throughout that 
    \begin{enumerate}
        \item $E$ and $E^{(K)}$ have good ordinary reduction at $p$;
        \item  $\op{corank}_{\Z_p} \op{Sel}_{p^\infty}(E/\Q)=0$;        \item $E(K)[p]=0$;
        \item $\Sh(E/K)[p^\infty]$ is finite.
    \end{enumerate}

\end{ass}
It follows from the structure theory of $\Lambda_{\op{cyc}}$-modules that $\mu_p(E/K_{\op{cyc}})=0$ and $\lambda_p(E/K_{\op{cyc}})=1$ if and only if as a $\Z_p$-module,
\[\op{Sel}_{p^\infty}(E/K_{\op{cyc}})\simeq \left(\Q_p/\Z_p\right)\oplus A,\] for some finite abelian group $A$. Let $(a,b)\in \mathbb{P}^1(\Z_p)$ and $K_{a,b}$ the associated $\Z_p$-extension of $K$. Let $K_{a,b}^{(n)}$ be the $n$-th layer of $K_{a,b}/K$, i.e., $K_{a,b}^{(n)}$ is the subextension for which $[K_{a,b}^{(n)}:K]=p^n$. Identify the Iwasawa algebra $\Lambda_{a,b}$ with a formal power series ring $\Z_p\llbracket T\rrbracket$, with respect to some choice of topological generator of $\Gamma_{a,b}$. Denote by $f_{a,b}(T)$ the characteristic element associated with $\op{Sel}_{p^\infty}(E/K_{a,b})$, with the understanding that $f_{a,b}(T)$ is $0$ if $\op{Sel}_{p^\infty}(E/K_{a,b})$ is not cotorsion over $\Lambda_{a,b}$. Set $\mu_p(E/K_{a,b})$ (resp. $\lambda_p(E/K_{a,b})$) the $\mu$-invariant (resp. $\lambda$-invariant) associated to $\op{Sel}_{p^\infty}(E/K_{a,b})$. 

\begin{lemma}\label{lemma V^G=0 implies V=0}    Consider a finite dimensional $\F_p$-vector space $V$ and the finite $p$-group $G$ acting on $V$. If $V^G=0$, then it follows that $V=0$.\end{lemma}

\begin{proof}
The result follows from Nakayama's lemma, but we give a more elementary proof. Consider the relation $\# V = \sum_i \# \mathcal{O}_i$, where $\mathcal{O}_i$ runs through the $G$-orbits of $V$. Since it is assumed that $V^G = 0$, it follows that $\{0\}$ is a $G$-orbit consisting of just one element, and all other $G$-orbits have sizes that are powers of $p$. This implies in particular that $\# V\equiv 1\mod{p}$. Therefore, $\#V=1$, i.e., $V=0$.
\end{proof}

\begin{lemma}\label{trivial lemma}
    Assume that $E(K)[p^\infty]=0$. Then, it follows that $E(K_{a,b})[p^\infty]=0$. 
\end{lemma}
\begin{proof}
    It suffices to show that for all $n$, $E(K_{a,b}^{(n)})[p^\infty]=0$. Let $\Gamma_n:=\op{Gal}(K_{a,b}^{(n)}/K)$. We set $V:=E(K_{a,b}^{(n)})[p]$ and note that 
    \[V^{\Gamma_n}=E(K)[p]=0.\] Lemma \ref{lemma V^G=0 implies V=0} then implies that $V=0$. Since $E(K_{a,b}^{(n)})[p]=0$, it is clear that $E(K_{a,b}^{(n)})[p^\infty]=0$. This concludes the proof.
\end{proof}

\begin{proposition}\label{propn on H10 for finite layers}Let $K_{a, b}$ be a $\Z_p$-extension of $K$ and assume that 
    \begin{enumerate}
        \item $\op{rank} E(K)=1$;
        \item $\op{Sel}_{p^\infty}(E/K_{a,b})$ is cotorsion as a $\Lambda_{a,b}$-module;
        \item $\mu_p(E/K_{a,b})=0$ and $\lambda_p(E/K_{a,b})=1$.
        \end{enumerate}
        For all $n\in \Z_{\geq 0}$, the extension $K_{a,b}^{(n)}/\Q$ is integrally diophantine. Thus Conjecture \ref{denef lipchitz conjecture} is satisfied for $K_{a,b}^{(n)}$. In particular, Hilbert's tenth problem has a negative answer for $K_{a,b}^{(n)}$.
\end{proposition}
\begin{proof}
    It follows from Lemma \ref{basic lemma on lambda and corank} that 
    \[\op{Sel}_{p^\infty}(E/K_{a,b})\simeq \left(\Q_p/\Z_p\right)\oplus A,\] where $A$ is a finite $p$-group. It is easy to see (from the inflation-restriction sequence) that the kernel of the natural map 
    \[\alpha_n: \op{Sel}_{p^\infty}(E/K_{a,b}^{(n)})\rightarrow \op{Sel}_{p^\infty}(E/K_{a,b})\] injects into $H^1(K_{a,b}/K_{a,b}^{(n)},E(K_{a,b})[p^\infty])$. It follows from Lemma \ref{trivial lemma} that $E(K_{a,b})[p^\infty]=0$. Thus, $\alpha_n$ is injective. Therefore, it follows that 
    \[\op{corank}_{\Z_p}\left( \op{Sel}_{p^\infty}(E/K_{a,b}^{(n)})\right)\leq \op{corank}_{\Z_p}\left( \op{Sel}_{p^\infty}(E/K_{a,b})\right)=1.\]
    In particular, this implies that \[\op{rank} E(K_{a,b}^{(n)})\leq 1.\] On the other hand, since $\op{rank} E(K)=1$, it follows that for all $n$, 
    \[\op{rank} E(K_{a,b}^{(n)})=\op{rank}E(K)>0.\]
    Thus by Theorem \ref{shlap 1}, we have that $K_{a,b}^{(n)}/K$ is integrally diophantine. Since $K/\Q$ is abelian, it is an integrally diophantine extension. Thus, it follows that $K_{a, b}^{(n)}/\Q$ is integrally diophantine, thus verifying Conjecture \ref{denef lipchitz conjecture}. In particular, for all $n$, Hilbert's tenth problem has a negative answer for $K_{a,b}^{(n)}$.
\end{proof}

We shall give conditions for the assumptions of Proposition \ref{propn on H10 for finite layers} to hold. Let $M:=\op{Sel}_{p^\infty}(E/K_\infty)^\vee$, viewed as a module over $\Lambda_\infty=\Z_p\llbracket X, Y\rrbracket$. Following \cite{gajek2024conjecture}, we introduce a condition on the algebraic structure of $M$.

\begin{definition}The module $M$ is said to satisfy (\mylabel{hyp:SC}{\textbf{S-C}}) if it is a direct sum of cyclic torsion $\Lambda_\infty$-modules.
\end{definition}

The module $M$ over $\Z_p\llbracket X, Y\rrbracket$ is said to be \emph{pseudonull} if for all prime ideals $\mathfrak{p}$ of $\Z_p\llbracket X, Y\rrbracket$ of height $\leq 1$, one has that $M_{\mathfrak{p}}=0$. The maximal pseudonull submodule of $M$ is denoted by $M_o$. 

\begin{proposition}\label{Mo=0}
    With respect to notation above, we have that $M_o=0$.
\end{proposition}
\begin{proof}
    The result follows from \cite[Proposition 6.1]{KLRnoncom}.
\end{proof}

\begin{lemma}
\label{lemma:structure}
Suppose that $M$ satisfies \eqref{hyp:SC}.
Then there exist principal ideals $I_1,\dots, I_m$ of $\Lambda_\infty$ such that $M\simeq \bigoplus_{i=1}^m\Lambda_\infty/I_i$.
\end{lemma}
\begin{proof}
    The result follows directly from Proposition \ref{Mo=0} and \cite[Theorem 2.4]{gajek2024conjecture}.
\end{proof}
\begin{definition}
Suppose that $M$ satisfies \eqref{hyp:SC}.
The \emph{characteristic ideal} of $M$ is defined to be
\[
\op{char}_{\Lambda_\infty}(M)=\prod_{i=1}^m I_i,
\]
where $I_i$'s are the principal ideals prescribed in Lemma \ref{lemma:structure}.
\end{definition}

We recall an explicit condition for \eqref{hyp:SC} to be satisfied. 
\begin{proposition}
\label{prop:SC-ord}
With respect to notation above, suppose that $\lambda(\Sel_{p^\infty}(E/K_\cyc)^\vee)\le 1$ and $\mu(\Sel_{p^\infty}(E/K_\cyc)^\vee)=0$, then $\Sel_{p^\infty}(E/K_\infty)^\vee$ satisfies \eqref{hyp:SC}.
\end{proposition}

\begin{proof}
    The result is \cite[Proposition 5.1]{gajek2024conjecture}.
\end{proof}

\begin{definition}The condition (\mylabel{hyp:Na}{\textbf{N-a}}) is satisfied if for all primes $v|p$ of $K$, we have that $p\nmid \#\widetilde{E}(\kappa_v)$. 
\end{definition}

\begin{proposition}
\label{prop:control-ord}
Suppose that $E$ satisfies \eqref{hyp:Na}. Then the natural restriction map then induces an isomorphism
\[
\Sel_{p^\infty}(E/K_{a,b})\simeq \Sel_{p^\infty}(E/K_\infty)^{H_{a,b}}.
\]
\end{proposition}

\begin{proof}
    The result is \cite[Proposition 3.2]{gajek2024conjecture}.
\end{proof}

\begin{theorem}\label{main thm 1}
Let $(a,b)$ and $(1, 0)$ be $p$-congruent, i.e., $a\equiv 1\pmod{p}$ and $b\equiv 0\pmod{p}$. Suppose that the following conditions hold
    \begin{enumerate}
         \item \eqref{hyp:Na} is satisfied, 
        \item $\op{Sel}_{p^\infty}(E/K_{\op{cyc}})$ is a cofinitely generated and cotorsion $\Lambda_{\op{cyc}}$-module.
        \item $\mu_p(E/K_{\op{cyc}})=0$ and $\lambda_p(E/K_{\op{cyc}})=1$.
    \end{enumerate} 
    Then, the following assertions hold:
    \begin{enumerate}
        \item $\op{Sel}_{p^\infty}(E/K_{a,b})$ is a cofinitely generated and cotorsion $\Lambda_{a,b}$-module,
        \item $\mu_p(E/K_{a,b})=0$ and $\lambda_p(E/K_{a,b})=1$.
        \item For all $n\in \Z_{\geq 0}$, the Conjecture \ref{denef lipchitz conjecture} is satisfied for $K_{a,b}^{(n)}$. In particular, Hilbert's tenth problem has a negative answer for $K_{a,b}^{(n)}$.
    \end{enumerate} 
\end{theorem}
\begin{proof}
   Note that it follows from Proposition \ref{prop:SC-ord} that \eqref{hyp:SC} is satisfied. Lemma \ref{lemma:structure} then implies that $\op{Sel}_{p^\infty}(E/K_\infty)^\vee$ decomposes as a direct sum of cyclic modules as follows
\[
\op{Sel}_{p^\infty}(E/K_\infty)^\vee\simeq \bigoplus_{i=1}^m \Lambda_\infty/(g_i),
\]
for $g_1, \dots, g_m \in \Lambda_\infty$. We note that $\Lambda_\infty=\Z_p\llbracket X, Y\rrbracket$ and $\Lambda_{a,b}=\Z_p\llbracket X, Y\rrbracket/(\omega_{a,b})$, where
$\omega_{a,b}=(1+X)^a(1+Y)^b-1$. By Proposition~\ref{prop:control-ord},
\[
\Sel_{p^\infty}(E/K_\infty)^\vee_{H_{a,b}}\simeq \Sel_{p^\infty}(E/K_{a,b})^\vee\simeq \bigoplus_{i=1}^m \Lambda_\infty/(g_i,\omega_{a,b})\simeq \bigoplus_{i=1}^m \Lambda_{a,b}/(\pi_{a,b}(g_i)).
\]
We identify $\Lambda_{\op{cyc}}$ with $\Z_p\llbracket X, Y\rrbracket/(X)=\Z_p\llbracket Y\rrbracket$. Observe that
\[\omega_{a,b}=(1+X)^a(1+Y)^b-1=aX+bY+\emph{higher degree terms},\] and that there is a uniquely determined power series $g(Y)\in \Z_p\llbracket Y\rrbracket$ such that 
\[\omega_{a,b}(g(Y), Y)=0.\]
The function \[g(Y)=c_1 Y+c_2 Y^2+c_3 Y^3+\dots \] is constructed by inductively solving for $c_i$, noting that $c_1=-b/a$. There is a natural identification
\[\Lambda_{a,b}=\Z_p\llbracket X,Y\rrbracket/(\omega_{a,b})\xrightarrow{\sim} \Z_p\llbracket Y\rrbracket,\] where the isomorphism takes $X\mapsto g(Y)$ and $Y\mapsto Y$. Let $F(X,Y)=\prod_i g_i$ be a generator of the characteristic ideal of $\op{Sel}_{p^\infty}(E/K_\infty)$. Let $h(Y)$ and $h_{a, b}(Y)$ be generators of the characteristic ideals of $\op{Sel}_{p^\infty}(E/K_{\op{cyc}})$ and $\op{Sel}_{p^\infty}(E/K_{a,b})$ respectively. Note that $\pi_{1,0}(X)=0$ and $\pi_{1,0}(Y)=Y$. Thus, we find that 
\[
    h(Y)=F(0,Y)\text{ and } h_{a,b}(Y)=F(g(Y), Y).\]
    Write $h(Y)=a_0+a_1 Y+a_2Y^2+\dots$ and $h_{a,b}(Y)=b_0+b_1 Y+\dots$. By construction $g(Y)=-a^{-1}b Y+c_2 Y^2+c_3Y^3+\dots$, and note that $p|a^{-1}b$; we find that $a_0=b_0$ and $a_1\equiv b_1\pmod{p}$. Since $\mu_p(E/K_{\op{cyc}})=0$ and $\lambda_p(E/K_{\op{cyc}})=1$, we have $p\mid a_0$ and $p\nmid a_1$. Since $a_0=b_0$ and $a_1\equiv b_1\pmod p$, it follows that $p\mid b_0$ and $p\nmid b_1$. In particular, $h_{a,b}(T)\neq 0$, and hence $\op{Sel}_{p^\infty}(E/K_{a,b})$ is cotorsion. Moreover, the conditions $p\mid b_0$ and $p\nmid b_1$ show that $\mu_p(E/K_{a,b})=0$ and $\lambda_p(E/K_{a,b})=1$. This proves part (2). Part (3) then follows from Proposition \ref{propn on H10 for finite layers}.
\end{proof}
\begin{theorem}\label{main thm 2}
    Assume that we have the following conditions
     \begin{enumerate}
         \item \eqref{hyp:Na} is satisfied, 
        \item $\op{Sel}_{p^\infty}(E/K_{\op{cyc}})$ is a cofinitely generated and cotorsion $\Lambda_{\op{cyc}}$-module.
        \item $\mu_p(E/K_{\op{cyc}})=0$ and $\lambda_p(E/K_{\op{cyc}})=1$.
    \end{enumerate} 
    Then there exists exactly one line $(\bar{a}_0: \bar{b}_0)\in \mathbb{P}^1(\mathbb{F}_p)$ such that for all $(a,b)\in \mathbb{P}^1(\Z_p)$ with $(\bar{a}, \bar{b})\neq (\bar{a}_0: \bar{b}_0)$, 
    \begin{enumerate}
        \item $\op{Sel}_{p^\infty}(E/K_{a,b})$ is a cofinitely generated and cotorsion $\Lambda_{a,b}$-module,
        \item $\mu_p(E/K_{a,b})=0$ and $\lambda_p(E/K_{a,b})=1$.
         \item For all $n\in \Z_{\geq 0}$, the Conjecture \ref{denef lipchitz conjecture} is satisfied for $K_{a,b}^{(n)}$. In particular, Hilbert's tenth problem has a negative answer for $K_{a,b}^{(n)}$.
    \end{enumerate}
\end{theorem}
\begin{proof}
    Let $g_i$ be as in the proof of Theorem \ref{main thm 1} and write \[F(X,Y):=\prod g_i=\sum_{i,j}a_{i,j}X^iY^j.\] Consider as before the power series $\omega_{a,b}$ (but this time we do not assume that $a\equiv 1\pmod{p}$, $b\equiv 0\pmod{p}$). 

    We first consider the case where $a$ is a $p$-adic unit. As the cyclotmic line is already covered in Theorem \ref{main thm 1}, we can assume that $b$ is coprime to $p$ as well. As in the proof of Theorem \ref{main thm 1} there is a uniquely determined power series $g(Y)\in \Z_p\llbracket Y\rrbracket$ such that $\omega_{a,b}(g(Y), Y)=0$ and we find that $h_{a,b}(Y)=F(g(Y),Y)$. By construction, $g(Y)$ is divisible by $Y$ and the linear coefficient is $-b/a$ is a $p$-adic unit. We obtain
    \[\begin{split}& h_{a,b}(Y)=F(g(Y),Y)=a_{0,0}+(-a_{1,0}b/a+a_{0,1})Y+O(Y^2);\\
    &h(Y)=F(0,Y)=a_{0,0}+a_{0,1}Y+O(Y^2).\end{split}\]
    As $\mu_p(E/K_{\op{cyc}})=0$ and $\lambda_p(E/K_{\op{cyc}})=1$ we know that $a_{0,0}$ is divisible by $p$ and that $a_{0,1}$ is a $p$ adic unit. If $a_{1,0}$ is coprime to $p$, there exists exactly one value for $b\pmod p$ such that $(-a_{1,0}b/a+a_{0,1})$ is divisible by $p$, i.e. such that $\mu_p(E/K_{a,b})\neq 0$ or $\lambda_p(E/K_{a,b})\neq 1$. If $p$ divides $a_{1,0}$ then the coeffcient $(-a_{1,0}b/a+a_{0,1})$ is always coprime to $p$ and we obtain $\lambda_p(E/K_{a,b})=1$ and $\mu_p(E/K_{a,b})=0$ for all values of $(a,b)$. 

    It remains to consider the case $a=0$ and $b=1$, i.e. $\omega_{a,b}=Y$. In this case we obtain the anticyclotomic $\Z_p$-extension and the characteristic power series associated to the Selmer group is 
    \[F(X, 0)=a_{0,0}+Xa_{1,0}+O(X^2)\]
    Recall that $p$ divides $a_{0,0}$. We find that $p\nmid  a_{1,0}$ if and only if $\mu_p(E/K_{a,b})=0$ and $\lambda_p(E/K_{a,b})=1$.

    Therefore, we find that there is exactly one line $(\bar{a}_0, \bar{b}_0)\in \mathbb{P}^1(\F_p)$ such that $f_{a,b}\neq 0$, $\mu_p(E/K_{a,b})=0$ and $\lambda_p(E/K_{a,b})=1$ for all $(a,b)\not\equiv (\bar{a}_0, \bar{b}_0)\pmod{p}$. This proves parts (1) and (2). Part (3) then follows from the proof of Proposition \ref{propn on H10 for finite layers}.
\end{proof}
We now give conditions for the conditions of Theorem \ref{main thm 1} (or equivalently, Theorem \ref{main thm 2}) to be satisfied. Note that if $E$ is defined over $\Q$, it follows from the results of Kato and Rohrlich that (2) is satisfied. Recall that Proposition \ref{prop:selmer-decomposition-ord} asserts that 
\[
    \mathrm{Sel}_{p^\infty}(E/K_{\mathrm{cyc}}) \simeq \mathrm{Sel}_{p^\infty}(E/\mathbb{Q}_{\mathrm{cyc}}) \oplus \mathrm{Sel}_{p^\infty}(E^{(K)}/\mathbb{Q}_{\mathrm{cyc}}).
    \]
    Note that $\mu_p(E/K_{\op{cyc}})=0$ if and only if $\mathrm{Sel}(E/K_{\mathrm{cyc}})$ is cofinitely generated as a $\Z_p$-module. Likewise, $\mu_p(E/\Q_{\op{cyc}})=0$ (resp. $\mu_p(E^{(K)}/\Q_{\op{cyc}})=0$) if and only if $\mathrm{Sel}_{p^\infty}(E/\Q_{\mathrm{cyc}})$ (resp. $\mathrm{Sel}_{p^\infty}(E^{(K)}/\Q_{\mathrm{cyc}})$) is cofinitely generated as a $\Z_p$-module. Thus we find that 
    \[\mu_p(E/K_{\op{cyc}})=0\Leftrightarrow \mu_p(E/\Q_{\op{cyc}})=\mu_p(E^{(K)}/\Q_{\op{cyc}})=0.\]
    If the above equivalent conditions are satisfied, then, 
    \[\begin{split}\lambda_p(E/K_{\op{cyc}})=&\op{corank}_{\Z_p} \left(\mathrm{Sel}_{p^\infty}(E/K_{\mathrm{cyc}})\right) \\ =&\op{corank}_{\Z_p} \left(\mathrm{Sel}_{p^\infty}(E/\Q_{\mathrm{cyc}})\right)+\op{corank}_{\Z_p} \left(\mathrm{Sel}_{p^\infty}(E^{(K)}/\Q_{\mathrm{cyc}})\right).\end{split}\]
For condition (3), we shall require that $\op{corank}_{\Z_p} \op{Sel}_{p^\infty}(E/\Q)=1$, $\op{corank}_{\Z_p} \op{Sel}_{p^\infty}(E^{(K)}/\Q)=0$ and that the Iwasawa invariants are prescribed as follows
\[\begin{split}& \mu_p(E/\Q_{\op{cyc}})=\mu_p(E^{(K)}/\Q_{\op{cyc}})=0,\\
& \lambda_p(E/\Q_{\op{cyc}})=1\text{ and }\lambda_p(E^{(K)}/\Q_{\op{cyc}})=0.\end{split}\]

\begin{theorem}\label{main thm 3}
    Let $E_{/\Q}$ be an elliptic curve, $p$ an odd prime and $K$ be an imaginary quadratic field. Assume that the following conditions hold
    \begin{enumerate}
        \item $\op{corank}_{\Z_p} \op{Sel}_{p^\infty}(E/\Q)=1$ and $\op{corank}_{\Z_p} \op{Sel}_{p^\infty}(E^{(K)}/\Q)=0$;
        \item $E$ and $E^{(K)}$ have good ordinary reduction at $p$; 
        \item $\widetilde{E}(\F_p)[p]=0$ and $\widetilde{E}^{(K)}(\F_p)[p]=0$;
        \item $p\nmid \prod_{\ell\neq p} c_\ell(E/\Q)\times \prod_{\ell\neq p} c_\ell(E^{(K)}/\Q)$;
        \item the normalized $p$-adic regulator of $E$ over $\Q$ is a $p$-adic unit;
        \item $\Sh(E/\Q)[p^\infty]=0$ and $\Sh(E^{(K)}/\Q)[p^\infty]=0$.
    \end{enumerate}
    Let $(a_0,b_0)$ be a representative for the line excluded in Theorem \ref{main thm 2}. 
    Then, for all $(a,b)\not \equiv (a_0, b_0)\pmod{p}$, the following assertions hold:
    \begin{enumerate}
        \item $\op{Sel}_{p^\infty}(E/K_{a,b})$ is a cofinitely generated and cotorsion $\Lambda_{a,b}$-module,
        \item $\mu_p(E/K_{a,b})=0$ and $\lambda_p(E/K_{a,b})=1$.
        \item For all $n\in \Z_{\geq 0}$, the Conjecture \ref{denef lipchitz conjecture} is satisfied for $K_{a,b}^{(n)}$. In particular, Hilbert's tenth problem has a negative answer for $K_{a,b}^{(n)}$.
    \end{enumerate} 
\end{theorem}
\begin{proof}
    Note that the Selmer groups $\op{Sel}_{p^\infty}(E/\Q_{\op{cyc}})$ and $\op{Sel}_{p^\infty}(E^{(K)}/\Q_{\op{cyc}})$ are cofinitely generated and cotorsion over $\Lambda_{\op{cyc}}$. The kernel of the reduction map \[E(\Q_p)\rightarrow \widetilde{E}(\F_p)\]
    consists of the $\Z_p$-points on the formal group of $E$ at $p$, and is in particular torsion free. Condition (3) thus implies in particular that $E(\Q_p)[p]=0$ and $E^{(K)}(\Q_p)[p]=0$. Thus, we find that $E(\Q)[p]=0$ and $E^{(K)}(\Q)[p]=0$.

    We apply the two Euler-characteristic criteria separately. For $E/\Q$, condition (1) is the rank-one Selmer-corank hypothesis, while conditions (2)--(6) give all the remaining hypotheses of Corollary~\ref{cor ECF mulambda}. Hence
    \[
        \mu_p(E/\Q_{\op{cyc}})=0
        \quad\text{and}\quad
        \lambda_p(E/\Q_{\op{cyc}})=1.
    \]
    For $E^{(K)}/\Q$, condition (1) gives Selmer corank zero. Its $p$-adic height pairing is therefore vacuously nondegenerate, and conditions (2), (3), (4), and (6) give the other hypotheses of the rank-zero criterion, Corollary~\ref{cor:rank-zero}. Consequently,
    \[
        \mu_p(E^{(K)}/\Q_{\op{cyc}})=0
        \quad\text{and}\quad
        \lambda_p(E^{(K)}/\Q_{\op{cyc}})=0.
    \]
From the discussion in the above lines, it follows that 
\[\mu_p(E/K_{\op{cyc}})=0\text{ and } \lambda_p(E/K_{\op{cyc}})=1.\]
The assertions thus follow from Theorem \ref{main thm 2}.
\end{proof}

\par We conclude this section by illustrating Theorems \ref{main thm 1}, \ref{main thm 2} and \ref{main thm 3}.
\begin{example}\label{example 4.15}
Consider the elliptic curve $E/\Q$ with LMFDB label $58.a1$ and set
$p=17$. The curve $E$ has rank $1$ and analytic rank $1$. Moreover,
the following properties hold:
\begin{itemize}
    \item the Tamagawa product of $E$ is equal to $2$;
    \item $E$ has good ordinary reduction at $17$;
    \item $a_{17}(E)=-4$, and hence
    $\widetilde{E}(\F_{17})[17]=0$;
    \item the normalized $17$-adic regulator of $E$ is a
    $17$-adic unit;
    \item $E(\Q)[17]=0$;
    \item
    $\mu_{17}(E/\Q_{\op{cyc}})=0$ and
    $\lambda_{17}(E/\Q_{\op{cyc}})=1$.
\end{itemize}

Since $E$ has analytic rank $1$, the theorem of Gross--Zagier and
Kolyvagin implies that $\Sh(E/\Q)$ is finite. The $17$-adic regulator
is nonzero, so the $17$-adic height pairing is nondegenerate. We may
therefore apply Corollary \ref{cor ECF mulambda}. Since the Tamagawa
product is prime to $17$, $\widetilde{E}(\F_{17})[17]=0$, and the
Iwasawa invariants are $\mu_{17}=0$ and $\lambda_{17}=1$, it follows
that
\[
        \Sh(E/\Q)[17^\infty]=0.
\]

Now set $K=\Q(\sqrt{-1})$ and let
$E'=E^{(K)}=E^{(-1)}$. Then $E'$ has LMFDB label $464.f1$.
Since $-1$ is a square in $\Q_{17}$, the curves $E$ and $E'$ are
isomorphic over $\Q_{17}$. In particular, $E'$ has good ordinary
reduction at $17$ and
\[
        a_{17}(E')=a_{17}(E)=-4,
        \qquad
        \widetilde{E}'(\F_{17})[17]=0.
\]
The curve $E'$ has rank $0$ and analytic rank $0$, its Tamagawa
product is equal to $2$, and
\[
        \mu_{17}(E'/\Q_{\op{cyc}})
        =
        \lambda_{17}(E'/\Q_{\op{cyc}})
        =0.
\]
The analytic rank-zero condition and the theorem of Gross--Zagier and
Kolyvagin imply that $\Sh(E'/\Q)$ is finite. Since
$\mu_{17}(E'/\Q_{\op{cyc}})
=\lambda_{17}(E'/\Q_{\op{cyc}})=0$, the classical Euler
characteristic is a $17$-adic unit. The rank-zero Euler characteristic
criterion, Corollary~\ref{cor:rank-zero}, together with the fact that
the Tamagawa product is prime to $17$ and
$\widetilde{E}'(\F_{17})[17]=0$, therefore gives
\[
        \Sh(E'/\Q)[17^\infty]=0.
\]

Thus
\[
 \op{corank}_{\Z_{17}}\Sel_{17^\infty}(E/\Q)=1,
 \qquad
 \op{corank}_{\Z_{17}}\Sel_{17^\infty}(E'/\Q)=0,
\]
and all the hypotheses of Theorem \ref{main thm 3} are satisfied.
Consequently, apart from the single exceptional residue class of
lines prescribed by Theorem \ref{main thm 2}, the finite layers of
the corresponding $\Z_{17}$-extensions of
$K=\Q(\sqrt{-1})$ give negative answers to Hilbert's tenth problem.
\end{example}

\begin{example}\label{example 4.16}
Consider the curve with label $61.a1$ and $p=11$. Then we have 
\begin{itemize}
    \item the Tamagawa product is equal to $1$.
    \item $E/\Q$ is of rank $1$
    \item the analytic rank is $1$
    \item the $p$-adic regulator is coprime to $p$
    \item $a_p=-5$ is not congruent to $1$ modulo $p$
\end{itemize}
If we twist this curve by $-3$ we obtain the curve with label 549.c1. Setting $E':=E^{(-3)}$, we find that the pair $(E',p)$ has the properties
\begin{itemize}
    \item the Tamagawa product is equal to $2$.
    \item $E'/\Q$ is of rank $0$
    \item the analytic rank is $0$
    \item the $p$-adic regulator is coprime to $p$
    \item $a_p=5$ is not congruent to $1$ modulo $p$
    \item $\mu_{11}(E'/\Q_{\op{cyc}})=\lambda_{11}(E'/\Q_{\op{cyc}})=0$.
\end{itemize}
Moreover,
\[
    \mu_{11}(E/\Q_{\op{cyc}})=0
    \quad\text{and}\quad
    \lambda_{11}(E/\Q_{\op{cyc}})=1.
\]
Both curves have analytic rank at most one, so the theorem of Gross--Zagier and Kolyvagin gives the finiteness of their Tate--Shafarevich groups. Corollary~\ref{cor ECF mulambda}, applied to $E$, and Corollary~\ref{cor:rank-zero}, applied to $E'$, now give
\[
    \Sh(E/\Q)[11^\infty]=\Sh(E'/\Q)[11^\infty]=0.
\]
Thus the strengthened hypotheses of Theorem~\ref{main thm 3} hold. Theorems \ref{main thm 1}, \ref{main thm 2} and \ref{main thm 3} thus apply in this case for $p=11$ and $K:=\Q(\sqrt{-3})$. This gives rise to many $\Z_{11}$-extensions of $K$ for which Hilbert's tenth problem has a negative answer.
\end{example}

\section{Density results}\label{s 5}
\par In this section, we prove the unconditional $p=3$ part of Theorem~\ref{lastthm} and the conditional Kriz--Li parts for $p=11,13,31,37$. The idea is to fix $(E,p)$ and vary $d$ over negative squarefree integers. Recall that $E^{(d)}:=E^{(K_d)}$, where $K_d:=\Q(\sqrt{d})$. Thus, we assume that the following conditions are satisfied for $E$:
\begin{enumerate}
    \item $E$ has good ordinary reduction at $p$;
    \item $\op{corank}_{\Z_p}\op{Sel}_{p^\infty}(E/\Q)=1$;
    \item the prime $p$ is non-anomalous for $E$;
    \item $p\nmid \prod_{\ell\neq p} c_\ell(E/\Q)$;
    \item the normalized $p$-adic regulator of $E$ over $\Q$ is a $p$-adic unit;
    \item $\Sh(E/\Q)[p^\infty]=0$.
\end{enumerate}

These conditions imply that $\op{Sel}_{p^\infty}(E/\Q_{\op{cyc}})$ is cotorsion over $\Lambda_{\op{cyc}}$ and moreover, the Iwasawa invariants are given by
\[
\mu_p(E/\Q_{\op{cyc}})=0\text{ and }\lambda_p(E/\Q_{\op{cyc}})=1.
\]
This follows from the formula for the truncated Euler characteristic, cf. Corollary~\ref{cor ECF mulambda}. As $d$ ranges over all negative squarefree integers, $K_d$ ranges over all imaginary quadratic fields. We wish to show that there is a positive proportion of such values $d$ for which $(K_d,p)$ satisfies \eqref{hyp:h10gen}. This follows as a consequence of Theorem~\ref{main thm 3} by showing that there is a positive density set of $d$ for which the following assertions hold for $E^{(d)}$:
\begin{enumerate}
    \item $E^{(d)}$ has good ordinary reduction at $p$;
    \item $\op{corank}_{\Z_p}\op{Sel}_{p^\infty}(E^{(d)}/\Q)=0$;
    \item $\widetilde{E}^{(d)}(\F_p)[p]=0$;
    \item $p\nmid \prod_{\ell\neq p} c_\ell(E^{(d)}/\Q)$;
    \item $\Sh(E^{(d)}/\Q)[p^\infty]=0$.
\end{enumerate}

The last condition is not a consequence of analytic rank zero. The theorem of Gross--Zagier and Kolyvagin gives only the finiteness of $\Sh(E^{(d)}/\Q)$. To record precisely the additional input needed below, let $\mathcal D$ be a positive-density family of negative squarefree twisting parameters and put
\[
\eta_{E,p}(\mathcal D):=
\liminf_{X\to\infty}
\frac{\#\{d\in\mathcal D:\ |d|<X,\ \Sh(E^{(d)}/\Q)[p^\infty]=0\}}
{\#\{d\in\mathcal D:\ |d|<X\}}.
\]
We refer to the assertion $\eta_{E,p}(\mathcal D)>0$ as the $p$-primary Tate--Shafarevich density hypothesis for $\mathcal D$. The positive-density conclusions obtained below from Kriz--Li are conditional on this hypothesis. The argument for $p=3$ instead produces twists with trivial $3$-Selmer group and therefore proves the required Tate--Shafarevich vanishing unconditionally.

\subsection{Density results via the method of Kriz and Li}

First, we apply the results of Kriz and Li \cite{kriz-li} to obtain conditional density theorems for $p=11,13,31,37$. Their results supply positive-density families of analytic rank-zero twists, but they do not by themselves imply the vanishing of the $p$-primary Tate--Shafarevich groups.

\begin{lemma}\label{lem:density-kriz-li}
Let $E/\Q$ be an elliptic curve, and let $p\geq 5$ be a prime such that
$E(\Q)[2]=\{0\}$. Let $N$ be the conductor of $E$ and let $K_0$ be
an imaginary quadratic field. Assume that the following conditions hold:
\begin{enumerate}
    \item $\op{corank}_{\Z_p}\Sel_{p^\infty}(E/\Q)=1$ and $\Sh(E/\Q)[p^\infty]=0$;
    \item $E$ has good ordinary reduction at $p$;
    \item the prime $p$ is non-anomalous for $E$;
    \item $p\nmid\prod_{\ell\neq p}c_\ell(E/\Q)$;
    \item the normalized $p$-adic regulator of $E$ over $\Q$ is a
    $p$-adic unit;
    \item $2$ splits in $K_0$;
    \item all primes dividing $pN$ split in $K_0$;
    \item let $P$ be the Heegner point for $E/K_0$. Let
    $E^{\mathrm{ns}}(\F_2)$ be the set of non-singular points of $E$
    over $\F_2$ and let $\log_{\omega_E}$ be the logarithm associated
    to the differential $\omega_E$. Then
    \[
        \frac{|E^{\mathrm{ns}}(\F_2)|\log_{\omega_E}(P)}{2}
        \not\equiv 0\pmod 2.
    \]
\end{enumerate}

Let $S$ be the set of primes $\ell$ coprime to $2N$ such that
\begin{enumerate}
    \item $\ell$ splits in $K_0$;
    \item $\ell$ is a square modulo every prime dividing $pN$;
    \item $\ell\equiv
        \begin{cases}
            1\pmod 4,&\text{if }2\nmid N,\\
            1\pmod 8,&\text{if }2\mid N;
        \end{cases}$
    \item $\op{Frob}_\ell\in\op{Gal}(\Q(E[2])/\Q)$ has order $3$.
\end{enumerate}
Let $E'$ be the twist of $E$ by $dd_{K_0}$, where $d>0$ is squarefree
and only divisible by primes in $S$. If, in addition,
\[
    \Sh(E'/\Q)[p^\infty]=0,
\]
then $(E,E',p)$ satisfies the hypotheses of
Theorem~\ref{main thm 3}. Consequently, the conclusions of
Theorems~\ref{main thm 1} and \ref{main thm 2} apply. In particular,
\eqref{hyp:h10gen} is satisfied for
\[
        K=\Q\bigl(\sqrt{dd_{K_0}}\bigr).
\]
\end{lemma}

\begin{proof}
We need to check the assumptions of Theorem~4.14. We first verify the
condition required in order to apply \cite[Theorem~4.3]{kriz-li}. Let
$\ell\mid d$. If $q$ is an odd prime dividing $N$, condition~(2) in
the definition of $S$ gives $\left(\frac{\ell}{q}\right)=1$.
Since $\ell\equiv 1\pmod 4$, quadratic reciprocity gives $\left(\frac{q}{\ell}\right)=1$.
Moreover, $\left(\frac{-1}{\ell}\right)=1$. If $2\mid N$, then by condition~(3) we have
$\ell\equiv1\pmod 8$, and hence $\left(\frac{2}{\ell}\right)=1$.
It follows that $\left(\frac{-N}{\ell}\right)=1$ for every prime $\ell\mid d$. Since $d$ is squarefree, the Chinese
remainder theorem shows that $-N$ is a square modulo $d$. We therefore
obtain from \cite[Theorem~4.3]{kriz-li} that $E'$ has algebraic and
analytic rank zero.

The result of Gross--Zagier and Kolyvagin \cite{gross-zagier,kolyvagin}
(see also \cite[Theorem~2.4]{Bur22}) implies that
$\Sh(E'/\Q)$ is finite. Consequently,
$\op{corank}_{\Z_p}\Sel_{p^\infty}(E'/\Q)=0$. This yields the
Selmer-corank assertion in condition~(1) of Theorem~\ref{main thm 3}.
The additional hypothesis on $E'$ and the vanishing hypothesis on $E$
in condition~(1) of the present lemma together give condition~(6) of
Theorem~\ref{main thm 3}. Since every prime dividing $d$ is a square
modulo $p$ and $p$ splits in $K_0$, the integer $dd_{K_0}$ is a
square modulo $p$. Hence $E$ and $E'$ are isomorphic over $\Q_p$,
which implies conditions~(2) and~(3).

Note that all primes of bad reduction split in
$\Q(\sqrt{dd_{K_0}})$. Thus, the Tamagawa factors at these primes
remain the same. At all primes dividing $dd_{K_0}$, the curve $E$
has good reduction and $E'$ has additive reduction. Thus, the
Tamagawa factors at these primes are bounded by $4$. As $p$ is odd,
the whole Tamagawa product is not divisible by $p$, which proves
condition~(4). The regulator condition~(5) does not depend on $E'$
and is therefore satisfied. Thus all six conditions of
Theorem~\ref{main thm 3} hold.
\end{proof}

\begin{lemma}\label{computation-density} Let $k$ be the number of distinct prime factors dividing $N$. Let $\mathfrak{d}(S)$ denote the Dirichlet density of $S$. If \[\op{Gal}(\Q(E[2])/\Q)\cong \Z/3\Z,\] we have that
    \[\mathfrak{d}(S)=\begin{cases}
        \frac{2}{3}\frac{1}{2^{k+2}} & \text{ if }K_0\neq \Q(\sqrt{-1});\\
        \frac{2}{3}\frac{1}{2^{k+1}} & \text{ if }K_0=\Q(\sqrt{-1}).
    \end{cases}\]

    If $\op{Gal}(\Q(E[2])/\Q)\cong S_3$,  we obtain a density of $\mathfrak{d}(S)=\frac{2}{3}\frac{1}{2^{k+1}}$.
\end{lemma}
\begin{proof} \par We begin by considering the case when $\op{Gal}(\Q(E[2]/\Q)\cong \Z/3\Z$. Let $L=K_0(\sqrt{-1},\sqrt{q_1},\dots,\sqrt{q_k})$, where the $q_i$ are the prime factors of $N$. As all the primes dividing $N$ are split in $K_0$, we see that $d_{K_0}$ is coprime to $q_i$ for $1\le i\le k$. Thus, $[L:\Q]=2^{k+1}$ if $\sqrt{-1}\in K_0$ and $[L:\Q]=2^{k+2}$ otherwise. Therefore, the set of primes that are totally split in $L/\Q$ has density $\frac{1}{2^{k+1}}$ if $K_0=\Q(\sqrt{-1})$ and equal to $\frac{1}{2^{k+2}}$ otherwise. The density of the primes having Frobenius of order $3$ in $\op{Gal}(\Q(E[2])/\Q)$  is equal to $\frac{2}{3}$. As $L\cap \Q(E[2])=\Q$, the Tchebotarev density of $S$ is just the product of the two densities we computed which gives the desired result.
    
    \par It remains to consider the case that $\op{Gal}(\Q(E[2])/\Q)\cong S_3$. The extension $\Q(E[2])$ contains the quadratic extension $\Q(\sqrt{\Delta(E)})$, where $\Delta(E)$ is the discriminant of $E$. As $\Q(\sqrt{\Delta(E)})\subset \Q(\sqrt{-1},\sqrt{q_1},\dots ,\sqrt{q_k})$, we see that every prime that splits in $\Q(\sqrt{-1},\sqrt{q_1},\dots ,\sqrt{q_k})$, also splits in $\Q(\sqrt{\Delta})/\Q$. In particular, the Frobenius in $\op{Gal}(\Q(E[2])/\Q)$ is either trivial or of order $3$. Thus, the density of $S$ is equal to $\frac{2}{3\cdot2^{k+1}}$
\end{proof}
\begin{remark}
Kriz and Li give a list of imaginary quadratic fields and elliptic curves satisfying the rank assertion in (1), (6), (8), and (7) for all primes that divide $N$. Then there is a positive-density set of primes $p$ that split in $K_0$ and satisfy (2), (3), and (4). The regulator condition (5) and the vanishing $\Sh(E/\Q)[p^\infty]=0$ in (1) must be checked separately for the fixed curve.
\end{remark}

\begin{example}\label{important example}
Consider the curve $E$ with Cremona label 37a1 and the imaginary quadratic field $K_0=\Q(\sqrt{-7})$. We are looking for a prime $p$ such that $(E,K_0,p)$ satisfies the conditions of Lemma \ref{lem:density-kriz-li}. As $E/\Q$ has trivial torsion the condition $E(\Q)[2]=0$ is satisfied. Furthermore $E$ has rank $1$ over $\Q$ such that condition (1) is satisfied. As $2$ splits in $K_0$ condition (6) is satisfied. The only prime dividing $N$ is $37$ which splits in $K_0$. Kriz and Li checked that condition (8) is satified (see \cite[Example 6.1]{kriz-li}). Thus, we need to find a prime $p$ with the following properties
\begin{itemize}
    \item $E$ has good ordinary reduction at $p$
    \item $p$ is non-anomalous for $E$
    \item $p\nmid \prod_{l\neq p}c_l(E/\Q)$.
    \item the normalized $p$-adic regulator of $E$ over $\Q$ is a $p$-adic unit.
    \item $p$ splits in $K_0$.
\end{itemize}

The smallest prime $p\ge 5$ that splits in $K_0$ is $11$. It follows from the LMFDB data that $p=11$ also satisfies all the other conditions. In particular, the Iwasawa invariants are
\[
    \mu_{11}(E/\Q_{\op{cyc}})=0
    \quad\text{and}\quad
    \lambda_{11}(E/\Q_{\op{cyc}})=1.
\]
Since $E$ has analytic rank one, Gross--Zagier and Kolyvagin imply that $\Sh(E/\Q)$ is finite; Corollary~\ref{cor ECF mulambda} then gives
\[
    \Sh(E/\Q)[11^\infty]=0.
\]
Thus, the triple $(E,K_0,p)$ satisfies indeed all assumptions of Lemma~\ref{lem:density-kriz-li}. The first few primes in the set $S$ are $\{53,617,641,1409,1453\}$.

Kriz and Li provide a table of pairs $(E,K_0)$ satisfying conditions (1),(6), (7) only for primes dividing $N$ and (8). The table below records these pairs of elliptic curves and auxiliary imaginary quadratic fields $K_0$. The third column lists one prime such that the triple $(E,K_0,p)$ satisfies conditions (1)--(8) of Lemma~\ref{lem:density-kriz-li}.
(compare also with the list in \cite[Example 6.1]{kriz-li}).
\newline
   \[ \begin{tabular}{c|c|c}
        E&$ d_{K_0}$ &p \\\hline
         37a1&-7&11\\
         43a1&-7&11\\
         88a1&-7&37\\
         91a1&-55&31\\
         92b1&-7&11\\
         123a1&-23&13\\
         123b1&-23&13\\
         131a1&-23&13\\
         141a1&-23&13\\
         141d1&-23&13\\
         148a1&-7&11\\
         
    \end{tabular}
    \]

For every row, the same rank-one Euler-characteristic check, using the displayed prime and the corresponding Iwasawa data, verifies the fixed-curve condition
\[
    \Sh(E/\Q)[p^\infty]=0.
\]
\end{example}

\begin{theorem}\label{lastthm1}
Fix $p\in\{11,13,31,37\}$. For a row $(E,K_0,p)$ of the table in Example~\ref{important example}, let $\mathcal D_{E,K_0,p}$ denote the positive-density family of negative fundamental discriminants supplied by the Kriz--Li construction used above. Suppose that at least one row with third entry $p$ satisfies
\[
    \eta_{E,p}(\mathcal D_{E,K_0,p})>0.
\]
Then $(K,p)$ satisfies \eqref{hyp:h10gen} for a positive density of imaginary quadratic fields $K$.
\end{theorem}
\begin{proof}
Fix such a row and retain those $d\in\mathcal D_{E,K_0,p}$ for which
\[
    \Sh(E^{(d)}/\Q)[p^\infty]=0.
\]
By the definition of $\eta_{E,p}$, these parameters form a positive-density subfamily. Lemma~\ref{lem:density-kriz-li} applies to every member of this subfamily and gives \eqref{hyp:h10gen}.
\end{proof}

Our computations lead us to make the following conjecture.
\begin{conjecture}
    For all primes large enough, there exists an elliptic curve $E/\Q$ and an imaginary quadratic field $K_0$ such that the triple $(E,K_0,p)$ satisfies the conditions of Lemma \ref{lem:density-kriz-li}.
\end{conjecture}
Indeed, given an odd prime number $p$, one expects that each of the conditions (1)--(5) be satisfied by a positive density of elliptic curves (ordered by their height \cite{Brumerinventiones}). Given an elliptic curve $E$ which satisfies (1)--(5), it is natural to expect that each of the conditions (6)--(8) be satisfied by a positive density of imaginary quadratic fields $K_0$. 

In applications it can be cumbersome to check the Heegner point conditions in Lemma \ref{lem:density-kriz-li}. If we work with elliptic curves that admit a $3$ isogeny, we obtain the following result. 
\begin{theorem}\label{thm:3-kriz-li}
Let $E/\Q$ be a semistable elliptic curve of conductor $N$ which admits a rational $3$-isogeny. Assume that $E$ has rank $1$, and let $p>3$ be a prime such that
\begin{enumerate}
    \item $\op{corank}_{\Z_p}\Sel_{p^\infty}(E/\Q)=1$ and $\Sh(E/\Q)[p^\infty]=0$;
    \item $E$ has good ordinary reduction at $p$;
    \item $p$ is non-anomalous for $E$;
    \item $p\nmid\prod_{\ell\neq p}c_\ell(E/\Q)$;
    \item the normalized $p$-adic regulator of $E$ over $\Q$ is a $p$-adic unit.
\end{enumerate}
Put
\[
 r=
 \begin{cases}
 0 & \text{if $N$ is odd},\\
 2 & \text{if $N$ is even},
 \end{cases}
\]
and let $m$ be the number of odd primes $\ell\neq3$ dividing $N$. For a prime $\ell$, put $q_\ell=\ell$ if $\ell$ is odd and $q_2=4$. Finally, set
\[
 \alpha_3(E)=
 \begin{cases}
 \dfrac{1}{3},
 & \text{if $3\nmid N$, or if $E$ has split multiplicative reduction at $3$},\\[6pt]
 \dfrac{1}{4},
 & \text{if $E$ has non-split multiplicative reduction at $3$}.
 \end{cases}
\]
Then there is a set $\mathcal D_{\mathrm{KL}}$ of negative fundamental discriminants of density at least
\[
 \frac{p}{2(p+1)}
 \frac{\alpha_3(E)}{2^{r+m}}
 \prod_{\substack{\ell\mid N\\ \ell\neq3}}
 \frac{q_\ell}{\ell+1}
\]
for which the analytic rank-zero and local conditions required in Theorem~\ref{main thm 3} hold. Put
\[
\mathcal D_{\mathrm{KL}}^{\Sh}
:=
\{d\in\mathcal D_{\mathrm{KL}}:\Sh(E^{(d)}/\Q)[p^\infty]=0\}.
\]
Then $(\Q(\sqrt d),p)$ satisfies \eqref{hyp:h10gen} for every $d\in\mathcal D_{\mathrm{KL}}^{\Sh}$. In particular, if the $p$-primary Tate--Shafarevich density hypothesis $\eta_{E,p}(\mathcal D_{\mathrm{KL}})>0$ holds, then the set of negative fundamental discriminants $d$ for which $(\Q(\sqrt d),p)$ satisfies \eqref{hyp:h10gen} has positive lower density. More precisely, its lower density is at least $\eta_{E,p}(\mathcal D_{\mathrm{KL}})$ times the displayed lower bound.
\end{theorem}

\begin{proof}
Let $N_1$ be the product of the primes at which $E$ has split multiplicative reduction and let $N_2$ be the product of the primes at which $E$ has non-split multiplicative reduction. Since $E$ is semistable, $N=N_1N_2$. We apply \cite[Theorem~9.5]{kriz-li} to the triple $(N_1,N_2,1)$.

We impose in addition the local condition $\psi_d(p)=1$. Since $p\nmid3N$, this can be incorporated into the congruence conditions used in the proof of \cite[Theorem~9.5]{kriz-li}: one adjoins $p$ to the modulus and restricts $d$ modulo $p$ to the nonzero quadratic residues. The valid-pair argument of \cite[Proposition~9.3]{kriz-li} is unchanged. The additional condition contributes the factor $p/(2(p+1))$ to the density.

We now use \cite[Proposition~9.7]{kriz-li}. If $3\nmid N$, or if $E$ has split multiplicative reduction at $3$, that proposition gives a proportion at least
\[
\frac{1}{3\cdot2^{r+m}}
\prod_{\substack{\ell\mid N\\ \ell\neq3}}
\frac{q_\ell}{\ell+1}
\]
of negative fundamental discriminants for which $E^{(d)}$ has analytic rank zero. If $E$ has non-split multiplicative reduction at $3$, the corresponding lower bound for negative rank-zero twists is
\[
\frac{1}{2^{r+m+2}}
\prod_{\substack{\ell\mid N\\ \ell\neq3}}
\frac{q_\ell}{\ell+1}.
\]
Together with the additional condition at $p$, these are precisely the density bounds stated above. Denote the resulting family by $\mathcal D_{\mathrm{KL}}$.

Let $d\in\mathcal D_{\mathrm{KL}}$ and put $K=\Q(\sqrt d)$. Since $E^{(d)}$ has analytic rank zero, the theorem of Gross--Zagier and Kolyvagin implies that $E^{(d)}(\Q)$ has rank zero and $\Sh(E^{(d)}/\Q)$ is finite. Consequently,
\[
\op{corank}_{\Z_p}\Sel_{p^\infty}(E^{(d)}/\Q)=0.
\]

Moreover, $\psi_d(p)=1$ means that $p$ splits in $K$. Hence $d$ is a square in $\Q_p^\times$, so $E$ and $E^{(d)}$ are isomorphic over $\Q_p$. Thus $E^{(d)}$ has good ordinary reduction at $p$ and is non-anomalous at $p$.

It remains to check the Tamagawa condition. If $\ell\mid N$ and $\ell\neq3$, the local conditions in \cite[Theorem~9.5]{kriz-li} ensure that $E^{(d)}$ has non-split multiplicative reduction at $\ell$, and hence its Tamagawa number is at most $2$. At a prime dividing $d$ but not $N$, the twist has additive reduction and its Tamagawa number is at most $4$. At $3$, if $3\mid N$, the possibilities occurring in \cite[Proposition~9.7]{kriz-li} give either the original multiplicative Tamagawa factor, a non-split multiplicative factor of order at most $2$, or an additive factor of order at most $4$. Since $p>3$ and the original Tamagawa product is prime to $p$, it follows that
\[
p\nmid\prod_{\ell\neq p}c_\ell(E^{(d)}/\Q).
\]

Thus the pair $(E,E^{(d)})$ satisfies conditions~(1)--(5) of Theorem~\ref{main thm 3}. If $d\in\mathcal D_{\mathrm{KL}}^{\Sh}$, the fixed-curve hypothesis in (1) above and the definition of $\mathcal D_{\mathrm{KL}}^{\Sh}$ give condition~(6) as well. Theorem~\ref{main thm 3} therefore proves \eqref{hyp:h10gen} for every such $d$. The final density assertion follows immediately from the definition of $\eta_{E,p}(\mathcal D_{\mathrm{KL}})$.
\end{proof}
\begin{remark}
The explicit expression in Theorem~\ref{thm:3-kriz-li} is the density of the underlying Kriz--Li family. The lower density of the subfamily for which \eqref{hyp:h10gen} is proved is obtained by multiplying it by $\eta_{E,p}(\mathcal D_{\mathrm{KL}})$.
\end{remark}
\begin{example}
Consider the curve with Cremona label 209a1 and $p=5$. The rank-one Euler-characteristic criterion verifies $\Sh(E/\Q)[5^\infty]=0$ from the corresponding Iwasawa and local data. Thus $(E,p)$ satisfies the conditions of the above theorem. The positive-density conclusion for this example remains conditional on $\eta_{E,5}(\mathcal D_{\mathrm{KL}})>0$.
\end{example}

\subsection{Density results for $p=3$}

\par In this section, we fix the prime $p=3$ and an elliptic curve $E_{/\Q}$ with conductor $N_E$ satisfying the following conditions 
\begin{enumerate}
    \item $E$ admits a $3$-isogeny over $\Q$,
    \item $E$ has good ordinary reduction at $3$, 
    \item $\op{corank}_{\Z_3} \op{Sel}_{3^\infty}(E/\Q)=1$, 
    \item $\widetilde{E}(\F_3)[3]=0$;
        \item $3\nmid \prod_{\ell\neq 3} c_\ell(E/\Q)$;
        \item the normalized $3$-adic regulator of $E$ over $\Q$ is a $3$-adic unit;
        \item $\Sh(E/\Q)[3^\infty]=0$.
\end{enumerate}

We wish to apply the results of Bhargava-Klagsburn-Oliver-Schnidman \cite{BKOS} to prove that there is a positive density of negative squarefree numbers $d$ for which that following conditions hold for $E^{(d)}$
\begin{enumerate}
    \item $\op{Sel}_{3^\infty}(E^{(d)}/\Q)=0$;
    \item $E^{(d)}$ has good ordinary reduction at $3$;
    \item $3$ is non-anomalous for $E^{(d)}$, i.e., $\widetilde{E}^{(d)}(\F_3)[3]=0$;
    \item $3\nmid \prod_{\ell\neq 3} c_\ell(E^{(d)}/\Q)$.
\end{enumerate}
\par Let $\phi: E\rightarrow \widehat{E}$ be a $3$-isogeny defined over $\Q$. Note that this induces a $3$-isogeny $\phi_d: E^{(d)}\rightarrow \widehat{E}^{(d)}$.
\begin{definition}Let $\ell$ be a prime, define the local Selmer ratio as follows
\[c_\ell(\phi_d):=\dfrac{\left|\widehat{E}^{(d)}(\Q_\ell)/\phi\left(E^{(d)}(\Q_\ell)\right)\right|}{|E^{(d)}[\phi](\Q_\ell)|}.\]
The \emph{global Selmer ratio} is defined as the product of all local Selmer ratios
\[c(\phi_d) := \prod_{\ell} c_\ell(\phi_d).\]
\end{definition}
We recall some of the results of \cite{BKOS}. A subset $\Sigma \subset \Q^*/\Q^{*2}$ is said to be determined by local conditions if it can be expressed as $\Sigma = \Q^*/\Q^{*2} \cap \prod_\ell \Sigma_\ell$, where each $\Sigma_\ell$ is a subset of $\Q_\ell^*/\Q_\ell^{*2}$ representing a local condition. If $\Sigma_\ell = \Q_\ell^*/\Q_\ell^{*2}$ for all but a finite number of places $\ell$, we then say that $\Sigma$ is defined by finitely many local conditions. The height of an element $d \in \Q^*/\Q^{*2}$ is defined as $H(d) := |d|$, and for $X>0$, set $\Sigma(X):=\{ d\in \Sigma \mid |d| < X \}$.
Define, for each $m \geq 0$, the subset  \[T_m(\phi) := \{d \in \Q^*/\Q^{*2} \mid |t(\phi_d)| = m \},\]
where $t(\phi_d) := \ord_3 \, c(\phi_d)$, and let $\mu(T_m(\phi))$ denote the density of $T_m(\phi)$ within $\Q^*/\Q^{*2}$. The set $T_m(\phi)$ is defined by local conditions at the primes $\ell$ that divide $6N\infty$.

\begin{theorem}\label{globalbounds}
Let $E$ be an elliptic curve over $\Q$ admitting a $3$-isogeny $\phi \colon E \to \widehat{E}$. Let $\Sigma$ be a subset of $T_0(\phi)$ defined by finitely many local conditions. Then the proportion of twists $E^{(d)}$ for which $\op{Sel}_3(E^{(d)})=0$ is at least $\frac{1}{2}\mu(\Sigma)$.
\end{theorem}
\begin{proof}
    The above result follows from \cite[Theorem 2.5]{BKOS} and its proof.
\end{proof}

We recall results in \emph{loc. cit.}, which provide a description for the local Selmer ratios $c_\ell(\phi)$. We start with $\ell=\infty$. 

\begin{proposition}\label{R}
With respect to notation above, we have that
\begin{equation*}c_\infty(\phi_d) = 
\begin{cases}
\frac{1}{3} & E^{(d)}[\phi](\mathbb{R}) \simeq \Z/3\Z; \\
1 & E^{(d)}[\phi](\mathbb{R}) = 0.
\end{cases}\end{equation*}
\end{proposition}

\begin{proposition}
    \label{cor:tamag}
For $\ell\notin \{3, \infty\}$ we have that $c_\ell(\phi_d) = c_\ell(\widehat{E}^{(d)})/c_\ell(E^{(d)})$.
\end{proposition}
\begin{proof}
This result is \cite[Corollary 10.3]{BKOS}.
\end{proof}

We now come to the main application of this section. 

\begin{theorem}\label{lastthm 2}
There is unconditionally a positive density of negative squarefree integers such that $(\Q(\sqrt d),3)$ satisfies \eqref{hyp:h10gen}.
\end{theorem}

\begin{proof}
The strategy is to show that there is an explicit elliptic curve with $3$-isogeny $\phi$ for which the set $T_0(\phi)$ has positive density. The values of $d$ shall range in a certain subset of $T_0(\phi)$ subject to conditions at $3$ and $\infty$.

\par Let $E_{/\Q}$ be the elliptic curve $y^2=x^3+216x-54$ (LMFDB label \href{https://www.lmfdb.org/EllipticCurve/Q/1216/o/3}{1216.o3}) and set $p:=3$. There is a cyclic $3$-isogeny $\phi:E\rightarrow\widehat{E}$ defined over $\Q$. The Iwasawa invariants are given by
\[
\mu_3(E/\Q_{\op{cyc}})=0
\quad\text{and}\quad
\lambda_3(E/\Q_{\op{cyc}})=1.
\]
Since $E$ has analytic rank $1$, the Gross--Zagier--Kolyvagin theorem implies that $\Sh(E/\Q)$ is finite. The normalized $3$-adic regulator is a unit, so the $3$-adic height pairing is nondegenerate; moreover, the non-anomalous condition implies $E(\Q)[3]=0$. Corollary~\ref{cor ECF mulambda} therefore applies, and the displayed Iwasawa invariants imply that
\[
\Sh(E/\Q)[3^\infty]=0.
\]

We find that the twist $E^{(-2)}$ (LMFDB label \href{https://www.lmfdb.org/EllipticCurve/Q/304/f/3}{304.f3}) has $\mu_3(E^{(-2)}/\Q_{\op{cyc}})=0$ and $\lambda_3(E^{(-2)}/\Q_{\op{cyc}})=0$. The Tamagawa product for all curves in the isogeny class \href{https://www.lmfdb.org/EllipticCurve/Q/304/f/}{304.f} is $1$. It follows from Proposition~\ref{cor:tamag} that $c_\ell(\phi_{-2})=1$ for $\ell\notin\{3,\infty\}$. We find that $c(\phi_{-2})=c_3(\phi_{-2})c_\infty(\phi_{-2})$. According to Proposition~\ref{R}, $c_\infty(\phi_{-2})\in\{1,1/3\}$. On the other hand, it follows from \cite[Theorem~10.5]{BKOS} that $c_3(\phi_{-2})\in\{1,3\}$. Therefore, $c(\phi_{-2})\in\{1/3,1,3\}$. On the other hand, it follows from \cite[Prop.~42(ii)]{BES} that
\[
t(\phi_{-2})\equiv\op{dim}_{\F_3}\op{Sel}_3(E^{(-2)}/\Q)=0\pmod 2.
\]
Thus we deduce that $c(\phi_{-2})=1$, i.e., $-2\in T_0(\phi)$.

\par We now introduce local conditions at $3$ and $\infty$. Consider the subset $T_0'(\phi)$ of $T_0(\phi)$ consisting of those $d$ for which $d\equiv1\pmod 3$ and $d<0$. The set $T_0'(\phi)$ is nonempty since $-2\in T_0'(\phi)$. The set $T_0'(\phi)$ is defined by finitely many congruence conditions (cf. \cite{BKOS}) and is nonempty. This implies that $T_0'(\phi)$ has positive density. Set
\[
\mathcal D_3:=\{d\in T_0'(\phi):\Sel_3(E^{(d)}/\Q)=0\}.
\]
By Theorem~\ref{globalbounds}, $\mathcal D_3$ has positive density. For every $d\in\mathcal D_3$, the Kummer exact sequence
\[
0\longrightarrow E^{(d)}(\Q)/3E^{(d)}(\Q)
\longrightarrow \Sel_3(E^{(d)}/\Q)
\longrightarrow \Sh(E^{(d)}/\Q)[3]
\longrightarrow 0
\]
shows that $E^{(d)}(\Q)$ has rank zero and that $\Sh(E^{(d)}/\Q)[3]=0$. Since every nonzero $3$-primary torsion group contains an element of order $3$, it follows that $\Sh(E^{(d)}/\Q)[3^\infty]=0$, and consequently $\Sel_{3^\infty}(E^{(d)}/\Q)=0$.

For $d\in T_0'(\phi)$, the prime $3$ splits in $\Q(\sqrt d)$ because $d\equiv1\pmod 3$. Since $\widetilde E(\F_3)[3]=0$, we have $\widetilde E^{(d)}(\F_3)[3]=0$. For every prime $\ell\neq3$, the Kodaira type of $E^{(d)}$ is $I_0$, $I_0^*$, $I_1$, or $I_1^*$ (cf. \cite[Table~3]{BKOS}); in particular,
\[
3\nmid\prod_{\ell\neq3}c_\ell(E^{(d)}/\Q).
\]
Moreover, $E$ and $E^{(d)}$ have the same reduction type at $3$, so $E^{(d)}$ has good ordinary reduction there. Thus all six hypotheses of Theorem~\ref{main thm 3} hold for every $d\in\mathcal D_3$, and
\[
\mu_3(E/\Q(\sqrt d)_{\op{cyc}})=0
\quad\text{and}\quad
\lambda_3(E/\Q(\sqrt d)_{\op{cyc}})=1.
\]
Theorems~\ref{main thm 1} and \ref{main thm 2} therefore show that $(\Q(\sqrt d),3)$ satisfies \eqref{hyp:h10gen} for every $d\in\mathcal D_3$. This proves the assertion unconditionally.
\end{proof}


\bibliographystyle{alpha}
\bibliography{references}
\end{document}